\documentclass[12pt,a4paper]{amsart}
\usepackage[margin=2.9cm]{geometry}             
                  
\usepackage{graphicx}
\usepackage{amssymb, enumerate,bm}
\usepackage[matrix,arrow,ps]{xy}
\usepackage{epstopdf}
\usepackage{paralist}
\newcommand{\real}{{\rm Re}\,}
\newcommand{\imag}{{\rm Im}\,}
\newcommand{\C}{\mathbb C}
\newcommand{\N}{\mathbb N}
\newcommand{\R}{\mathbb R}

\newcommand{\dop}[1]{\frac{\partial}{\partial #1}}

\newcommand{\vnorm}[1]{{\| #1 \|}}
\newcommand{\GL}{{\mathsf{GL}}}
\newcommand{\diffable}[1]{{C}^{#1}}
\DeclareMathOperator{\spanc}{span}
\newtheorem{theorem}{Theorem}[section]
\newtheorem{lemma}[theorem]{Lemma}
\newtheorem{cor}[theorem]{Corollary}
\newtheorem{prop}[theorem]{Proposition}

\theoremstyle{remark}
\newtheorem{remark}[theorem]{Remark}
\theoremstyle{example}

\theoremstyle{definition}
\newtheorem{defi}[theorem]{Definition}
\numberwithin{equation}{section}
\DeclareMathOperator{\Aut}{Aut}

\begin{document}

\begin{abstract}
We prove finite jet determination for
(finitely) smooth CR diffeomorphisms
of (finitely) smooth Levi degenerate hypersurfaces in
$\mathbb{C}^{n+1}$ by constructing 
generalized stationary discs glued to 
such hypersurfaces. 
\end{abstract}

\author{Florian Bertrand, Giuseppe Della Sala and Bernhard Lamel}
\title[Jet determination and generalized stationary discs]{Jet determination of smooth CR automorphisms\\ and generalized stationary discs}


\thanks{Research of the first author was supported by an URB grant and a long-term
faculty development grant both from the American University of Beirut.}
\thanks{Research of the second author was supported by an URB grant from the American University of Beirut and by the Center for Advanced Mathematical Sciences.}
\thanks{Research of the third author was supported by the Austrian Science Fund FWF, project AI1776 and the Qatar National Research Fund, NPRP 7-511-1-098.}
\maketitle 

\section{Introduction}

Let $M, M'\subset \C^{n+1}$ be $\diffable{\ell}$-smooth hypersurfaces. We 
recall that the complex tangent space $T_p^c M$, for $p\in M$, is defined
by $T_p^c M = T_p M \cap i T_p M$, and is the largest complex subspace of
$\C^{n+1}$ contained in $T_p M$.  
A map $H\colon M\to M'$ (of class $\diffable{1}$) is said to 
be a CR map if $H'(p)|_{T_p^c M}$ maps $T_p^c M$ into 
$T_p^c M'$ and is complex linear. We will only be concerned with 
germs of CR maps which are also diffeomorphisms (of some regularity). 

CR maps possess strong rigidity properties. We are interested mostly in 
one particular aspect of this rigidity here, namely the 
finite determination property. In the setting where 
$M$ and $M'$ are real-analytic, and $H\colon M\to M'$ extends to a germ 
of a biholomorphic map or is given by a formal power series
 at a point $p\in M$, this is usually phrased 
in terms of the finite jet determination property, and we know 
that $H$ is determined by finitely many of its derivatives at 
a point $p\in M$ in many circumstances (we refer the reader to the paper of 
Baouendi, Mir, and Rotschild \cite{bmr} as well as to 
Juhlin's paper \cite{ju} and the discussion of the literature therein). 

Actually, in the real-analytic setting one knows quite a bit more: Not 
only are (formal) biholomorphisms between 
sufficiently nondegenerate hypersurfaces determined by their jets, 
they can actually be reconstructed from their jets in an analytic manner 
(see e.g. the survey \cite{lasurvey} and \cite{la-ju}). One of the 
appealing parts of such so-called jet parametrizations is that 
they provide insight into structural properties of the automorphism 
groups of manifolds. However, they depend 
on jets, and therefore pointwise information. If one tries 
to study CR diffeomorphisms which are a priori only smooth (of some
regularity), these methods are not applicable.

Our goal in this paper is to study such (finitely) smooth CR automorphisms of (finitely) smooth
hypersurfaces in $\C^{n+1}$. We shall assume 
that our hypersurfaces, in suitable coordinates $(z,w) \in \C^n \times \C$, pass 
through $0\in \C^{n+1}$ and that their defining functions are perturbations 
of a {\em finite type model hypersurface} $S_P$ of the
form
\[ \real w = P_d (z,\bar z), \]
where $P$ is a weighted homogeneous polynomial. To be more exact, we 
consider (sufficiently smooth) perturbations of $S_P$ 
given by 
\[ \real w = P_d (z, \bar z) + O (d+1),\] 
where we endow $z$ with the weight $1$ and $w$ with the weight $(2d+1)/ 2d$. We can now state our main theorem as follows.

\begin{theorem}\label{thm:main} Let $P$ be a weighted homogeneous polynomial such that $S_P$ is generically Levi-nondegenerate and 
the set of Levi-degenerate points containing $0$ has dimension at most $2n-1$. Then there exists an $\ell \in \mathbb{N}$ such 
that for any allowable perturbation $M$ of $S_P$ in any neighbourhood $U$ of $0$, every local CR diffeomorphism of class $C^\ell$ 
of $M$ is determined by its $\ell$-jet: If $H\colon M \to M$ and $\tilde H\colon M \to M$ are CR diffeomorphisms of class $C^\ell$ with
$j^\ell_0 H = j^\ell_0 \tilde H$, then $H= \tilde H$.
\end{theorem}

In fact, Theorem~\ref{thm:main} holds 
under the less stringent (but more technical)
condition that ``there exists an allowable vector'' $v\in T_0^c S_P$; this 
condition is explained in Definition~\ref{defadm1}. In order to not 
duplicate formulations of theorems, we shall use the definition of an 
allowable vector as well as the condition 
that $M$ is an allowable deformation (which is 
discussed in Definition~\ref{defadm}); both conditions are geometric conditions
in a suitable sense to be defined below, in particular, they can be defined independently of coordinates.  We shall simply say 
{\em allowable hypersurface} from now on 
to indicate an allowable perturbation
based on a model hypersurface which 
possesses an allowable vector. 
Our main theorem then reads: 

\begin{theorem}\label{thm:main2}
Let $M$ be an allowable hypersurface. Then there exists an $\ell \in \mathbb{N}$
only depending on the associated 
model hypersurface such 
that every local CR diffeomorphism of class $C^\ell$ 
of $M$ is determined by its $\ell$-jet: If $H\colon M \to M$ and $\tilde H\colon M \to M$ are CR diffeomorphisms of class $C^\ell$ with
$j^\ell_0 H = j^\ell_0 \tilde H$, then $H= \tilde H$.
\end{theorem}

The number $\ell$ depends on the specific form of $P$ and can, in essence, be computed given $P$; we shall
give upper bounds on $\ell$ later. However, an especially interesting aspect of the current paper is its
application to the problem of unique determination of smooth diffeomorphisms of smooth hypersurfaces, in 
which case one can use known jet determination results in the formal setting which already provide 
ways to compute $j^\ell_0 H$ from $j^{p_0} H$ for $\ell \geq p_0$. To be precise,
we need to introduce some notation. 
We define, for $\ell \in \mathbb{N} \cup \{ \infty, \omega \}$ and 
for a CR manifold $M$ of class $C^k$ for
$k\geq \ell$,  the spaces
\[ 
\Aut^\ell (M,p) = \{ H\colon M \to M \colon H \text{ is a germ at } p \text { of a  CR diffeomorphism of class } C^\ell\}, \]
and for a smooth CR manifold $M$, 
\[\Aut^{f} (M,p) = \{ H\colon M \to M \colon H \text{ is a formal  CR diffeomorphism  of } M\}.
 \]
Here the space of  formal CR diffeomorphisms of $M$
is defined to be the space of formal 
power series maps $H\colon \C^{n+1} \to 
\C^{n+1}$ which have the 
property that they are formal 
biholomorphisms of the associated 
formal manifold (given by the 
ideal generated by the Taylor series 
of the defining equations of $M$), i.e. for
one (and hence every) defining function
$\varrho$ of $M$ and 
for one (and hence every) local parametrisation $\mathbb{R}^{2n+1} \supset U \ni x \mapsto Z(x) \in M$ we have that
for any $\ell \in \mathbb{N}$  it 
holds that  
$\varrho(H(Z(x)), \overline{H(Z(x))}) = O(|x|^{\ell+1})$. 

In particular, by definition we 
have natural maps 
\[ j^k_p \colon \Aut^\ell (M,p) \to G^k_p (\C^{n+1}), \, k\leq \ell \text{ and } j^k_p  \colon 
\Aut^f (M,p) \to G^k_p (\C^{n+1}), \, k \in \mathbb{N},    \]
into the jet group of order $k$ of germs 
of biholomorphisms at $p$. We know 
that if $M$ is {\em formally holomorphically nondegenerate} and {\em formally minimal} then, by \cite{la-ju}, 
the map $j^k_p$ is injective for $k$ 
large enough. Every allowable smooth 
hypersurface is, as the reader can easily convince herself or himself, formally 
holomorphically nondegenerate and formally nonminimal. Therefore there exists a smallest  number $k_0 (M)$ such that
for $k\geq k_0(M)$, the map $
j^k_p  \colon 
\Aut^f (M,p) \to G^k_p (\C^{n+1})$
is injective; in particular, the $k$-jets
of smooth CR diffeomorphisms of $M$ are uniquely
determined by their $p_0 (M)$ jets. 

 Theorem \ref{thm:main2} therefore has the following immediate corollary. 

\begin{cor}\label{cor:main3}
Let $M$ be an allowable smooth hypersurface. 
Then there exists an $\ell_0 \in \mathbb{N}$ such that for $\ell \geq \ell_0$ the map $j^{k_0(M)}_p \colon \Aut^\ell (M,p) \to G^{k_0 (M)}_p (\C^{n+1}) $
is injective. Furthermore $\ell_0$ depends only on the associated model $S_P$. 
\end{cor}

We would like to point out that 
Corollary~\ref{cor:main3} seems 
to be the first case of a 
jet determination result for {\em (finitely) smooth}
CR diffeomorphisms aside from the finitely nondegenerate case. The jet determination problem for real-analytic CR diffeomorphisms of real-analytic CR manifolds has been studied widely, see e.g.
  \cite{ch-mo, eb-la-za, la-mi, ko-me}. In 
  the smooth case, results have been restricted to the setting of {\em finitely nondegenerate} hypersurfaces 
  (see e.g. \cite{eb, eb-la, ki-za}). 
  Our approach is most akin to the use
  of extremal discs by Huang \cite{hu,hu3}. Let us give an outline of our approach. 

When studying the automorphisms of a geometric structure, 
it is often convenient to extend the action of these automorphisms 
to spaces of invariant objects, and study the transformation
properties of these invariant objects. In the study of real-analytic
CR manifolds,
a  suitable family of associated objects
is the family of Segre varieties. However,
these have the drawback that they really can only be defined for real-analytic or formal CR manifolds and thus becomes unavailable in the setting
of smooth CR manifolds. 

Our approach in this paper is to construct
another family of associated invariant
objects, namely generalized stationary discs, which we refer to as 
$k$-stationary discs. We  show that for allowable hypersurfaces in $\mathbb{C}^{n+1}$, one can invariantly
attach a finite-dimensional family of generalized stationary discs. 

This approach has been pioneered 
by the first and 
the second author for hypersurfaces in $\mathbb{C}^2$ in \cite{be-de1}, generalizing the 
notion  L. Lempert \cite{le} used in his study of the Kobayashi 
metric on strictly convex domains (see also \cite{hu, tu}). These classical stationary discs are special analytic discs, attached to hypersurfaces $M$ of $\mathbb C^{n+1}$, which admit a lift (with a pole of order at most $1$ at $0$) to the conormal bundle of $M$. The conormal bundle can be seen as a real $2n+2$-dimensional submanifod of $\mathbb C^{2n+2}$ and, as it turns out, it is totally real if $M$ is Levi-nondegenerate \cite{we}.   Consequently, if $M$ is Levi-nondegenerate the study of stationary discs falls into the framework developed in \cite{gl1, gl2,fo}, and indeed the first author and L. Blanc-Centi employed this method to construct stationary discs  in \cite{be-bl}, and used it to show finite determination of automorphisms.

If the Levi form degenerates at some points, the conormal bundle admits complex tangencies, and therefore the 
attachment of discs is more complicated.
We shall overcome this difficulty by 
constructing an associated 
circle bundle $\mathcal N^k S_P$ (a 
bundle over $S^1 \times S_P$ whose fiber at $(\zeta, p)$ is $\zeta^k N_p S_P$) whose CR 
singularities allows for attaching 
discs which pass through the singularity
with certain predescribed orders. 
Geometrically, one can think of this 
construction as allowing 
a higher winding of the conormal part of  the disc ($k$ instead of $1$
in the case of a classical stationary disc). Our theorem on the existence of discs is now as follows. 

 \begin{theorem} 
If $M$ is an admissible hypersurface, then there exists a $k_0 \in \mathbb{N}$
and a finite dimensional manifold of 
(small) $k_0$-stationary discs attached to M.
 \end{theorem}

Our approach is based on the Riemann-Hilbert problem which as we 
already pointed out is singular in our situation. The approach described above 
will allow that
the problem can be studied with tools of \cite{be-de2}. 

We note that  for $n=2$, we recover 
the results in \cite{be-de1}. However let us stress that we cannot
generalize  the methods used in \cite{be-de1}, where the results
 are achieved by using a rather \lq\lq ad hoc\rq\rq\ procedure. So 
 in this paper we develop  methods which allow to treat the more
 general setting,  and which we in addition believe to be more  geometric.
 
The paper is organized as follows. In Section 2, we describe the needed preliminaries. Section 3 is devoted to weighted homogeneous model hypersurfaces. In Section 4, we study the existence of generalized stationary discs attached to admissible hypersurfaces. Finally, Section 5 is devoted to the proof of  the finite jet determination theorems for CR diffeomorphisms.

\section{Preliminaries} In this section, we collect some 
standard notation and collect facts which we will need throughout the paper. 
We denote by $\Delta$ the  unit disc in $\C$ and by $b\Delta$ its boundary. We use coordinates
 $(z,w) \in \C^{n+1}$, where  $z=(z_1,\ldots, z_n)$ are the standard coordinates in $\C^{n}$. 

\subsection{Function spaces}

 Let  $k$ be an integer and let $0< \alpha<1$.
 We write $\mathcal C^{k,\alpha}=\mathcal C^{k,\alpha}(b\Delta,\R)$ for the space of real-valued functions  defined on $b\Delta$ of class 
$C^{k,\alpha}$. We equip the space $\mathcal C^{k,\alpha}$ with its usual norm
$$\vnorm{ v }_{\mathcal{C}^{k,\alpha}}=\sum_{j=0}^{k}\vnorm{ v^{(j)} }_\infty+
\underset{\zeta\not=\eta\in b\Delta}{\mathrm{sup}}\frac{\|v^{(k)}(\zeta)-v^{(k)}(\eta)\|}{|\zeta-\eta|^\alpha}$$
where $\vnorm{ v^{(j)} }_\infty=\underset{b\Delta}{\mathrm{max}}\vnorm{ v^{(j)} }$.

We define $\mathcal C_\C^{k,\alpha} = \mathcal C^{k,\alpha} + i\mathcal C^{k,\alpha} = \mathcal C^{k,\alpha}(b\Delta,\C)$. Therefore $v\in \mathcal C_\C^{k,\alpha}$ if and only if 
$\real v, \imag v \in \mathcal C^{k,\alpha}$. We endow 
the space $\mathcal C_\C^{k,\alpha}$  with the norm
$$\vnorm{ v }_{\mathcal{C}_{\C}^{k,\alpha}}=
\|\real v\|_{\mathcal{C}^{k,\alpha}}+\|\imag v\|_{\mathcal{C}^{k,\alpha}}.$$ 
We denote by $\mathcal A^{k,\alpha}$ the subspace of $\mathcal C_{\C}^{k,\alpha}$  of  functions $f$ which possess a continuous extension
 $F:\overline{\Delta}\rightarrow \C$, with $F$ 
 holomorphic on $\Delta$.

Let $m$ be an integer. We denote by $\mathcal A^{k,\alpha}_{0^m}$
the subspace of 
$\mathcal C_{\C}^{k,\alpha}$ of functions that can be written as $(1-\zeta)^m f$, with  $f\in \mathcal A^{k,\alpha}$. Note that since $\mathcal A^{k,\alpha}_{0^m}$ is not a closed subspace of $\mathcal C_{\mathbb C}^{k,\alpha}$, it is not a Banach space with the induced norm. Instead, we equip $\mathcal A^{k,\alpha}_{0^m}$ with the following norm 
\begin{equation}\label{eqnorm}
\|(1-\zeta)^m f\|_{\mathcal A^{k,\alpha}_{0^m}}
=\vnorm{ f }_{\mathcal{C}_{\C}^{k,\alpha}}
\end{equation}
which makes it  a Banach space, isomorphic to $\mathcal A^{k,\alpha}$.  Note that the inclusion of $\mathcal A^{k,\alpha}_{0^m}$ in $\mathcal A^{k,\alpha}$ is a bounded operator.

Finally, we denote by $\mathcal C_{0^m}^{k,\alpha}$ the subspace of $\mathcal C^{k,\alpha}$ of functions that can be written as $(1-\zeta)^m v$ with $v\in \mathcal C_\C^{k,\alpha}$. The space $\mathcal C_{0^m}^{k,\alpha}$ is equipped with the norm
$$\|(1-\zeta)^m f\|_{\mathcal C_{0^m}^{k,\alpha}}=\vnorm{ f }_{\mathcal C_\C^{k,\alpha}}.$$
Notice that $\mathcal C_{0^m}^{k,\alpha}$ is a Banach space. Denote by $\tau_m$ the map $\mathcal C_{0^m}^{k,\alpha} \to \mathcal C_{\mathbb C}^{k,\alpha}$ given by $\tau_m ((1 - \zeta)^m v ) = v$.
We recall the following lemma from \cite{be-de2}:

\begin{lemma}\label{subspaces}  Define the closed subspace $\mathcal R_m$ of $\mathcal C_{\mathbb C}^{k,\alpha}$ by 
$$\mathcal R_m = \{v \in \mathcal C_{\mathbb C}^{k,\alpha} \ | \ v(\zeta) = (-1)^m \zeta^{-m} \overline {v(\zeta)} \ \forall \  \zeta \in b\Delta \}.$$ Then
\begin{compactenum}[\rm (i.)] 
\item $\tau_m$ maps $\mathcal C_{0^m}^{k,\alpha}$ isomorphically to $\mathcal R_m$;
\item if $m = 2m'$ is even, the map $v \mapsto \zeta^{m'}v$ induces an isomorphism between $\mathcal R_m$ and $\mathcal R_0 = \mathcal C^{k,\alpha}$;
\item if $m = 2m'+1$ is odd, the map $v \mapsto \zeta^{m'}v$ induces an isomorphism between $\mathcal R_m$ and $\mathcal R_1$.
\end{compactenum}
\end{lemma}

\subsection{Partial indices and Maslov index}\label{secbirk}
We denote by $\GL_N(\C)$ the general linear group on $\C^N$. 
Let $G: b\Delta \to  \GL_N(\C)$ be a smooth map. 
We consider a Birkhoff factorization 
(see \cite{ve}) of $-\overline{G}^{-1}G$:
$$ -\overline{G(\zeta)}^{-1}G(\zeta)=
B^+(\zeta)
\begin{pmatrix}
	\zeta^{\kappa_1}& & & (0) \\ &\zeta^{\kappa_2} & & \\ & & \ddots & \\ (0)& & &\zeta^{\kappa_{N}}
\end{pmatrix}
% \left(\begin{array}{cccc}\zeta^{\kappa_1}& & & (0) \\ &\zeta^{\kappa_2} & & \\ & & \ddots & \\ (0)& & &\zeta^{\kappa_{N}}\end{array}\right)
B^-(\zeta)\, \ \ \text{ for all } \zeta\in b\Delta$$
where  $B^+\colon\bar{\Delta}\to \GL_N(\C)$ and 
$B^-\colon(\C \cup \infty)\setminus\Delta\to \GL_N(\C)$  
 are  smooth  maps, holomorphic on $\Delta$ and $\C \setminus \overline{\Delta}$ respectively. 
The integers $\kappa_1, \dots, \kappa_N$  are called the {\it partial indices} of 
$-\overline{G}^{-1}G$ and  their sum 
$\kappa:=\sum_{j=1}^N\kappa_j$ is called 
the {\it Maslov index} of $-\overline{G}^{-1}G$ and it is equal to 
the winding number  of the map $\zeta \mapsto \det\left(-\overline{G(\zeta)}^{-1}G(\zeta)\right)$ around the origin.

\subsection{$k_0$-stationary discs}
Let $S=\{r=0\}$ be a finitely smooth hypersurface defined in a neighborhood of the origin in 
$\C^{n+1}$. Let $k, k_0$ be integers and let $0<\alpha<1$. 
We recall that a holomorphic disc 
$f \in (\mathcal A^{k,\alpha})^{n+1}$ 
is {\it attached} to  $S$ if
$f(\zeta) \in S$ for all $\zeta \in b\Delta$. 
The following definition was given in  \cite{be-de1}:  
\begin{defi}\label{defstat}
A holomorphic disc $f \in (\mathcal A^{k,\alpha})^{n+1}$ attached to $S=\{r=0\}$ is said to be
 $k_0$-stationary if there exists a  continuous function $c\colon b\Delta
 \to\R\setminus\{0\}$ such that the map $\zeta \mapsto \zeta^{k_0} c(\zeta)\partial r(f(\zeta))$, defined on $b \Delta$, extends as a map in $(\mathcal A^{k,\alpha})^{n+1}$.
\end{defi}
The set of such discs is invariant under  CR diffeomorphisms.
\begin{prop}\label{statinvar}
Let $S \subset \C^{n+1}$ be a finitely smooth real hypersurface containing $0$. There exists a neighborhood $U$ of the origin in $\C^{n+1}$ such that if $H$ is a CR diffeomorphism of class $\mathcal{C}^{k+1}$ sending $S\cap U$ to a real hypersurface $S' \subset \C^{n+1}$ and $f\colon \Delta \rightarrow U$ is a $k_0$-stationary disc in $(\mathcal A^{k,\alpha})^{n+1}$ attached to $S$ then the disc $H\circ f$ extends as a  $k_0$-stationary disc in $(\mathcal A^{k,\alpha})^{n+1}$ attached to $S'$. 
\end{prop}

\begin{proof}
Using Theorem 6.2.2 in \cite{ber}, we write $W=\bigcup \varphi \left(\overline{\Delta}\right)$ where the union is taken over all analytic discs $f$ attached to $S$. The CR diffeomorphism $H$ of class $\mathcal C^{k+1}$ admits a local holomorphic extension 
$\widetilde H$ to $W$ continuous up to $S \cap W$. 
The image of any 
 $k_0$-stationary disc $f\in (\mathcal A^{k,\alpha})^{n+1}$ attached to $S$  is contained in $W$. The map 
$H\circ f \in \mathcal C_\C^{k,\alpha}$ defined on
$\partial \Delta$ therefore 
extends as $\widetilde H\circ f\in (\mathcal A^{k,\alpha})^{n+1}$. 
The rest of the proof is the same computation as given in  
\cite[Proposition 2.5]{be-de1}. 
\end{proof}
In our context, the following geometric version of Definition \ref{defstat} is more convenient to work with. 
\begin{defi}\label{defstat2}
A holomorphic disc $f \in (\mathcal A^{k,\alpha})^{n+1}$ attached to  $S=\{r=0\}$ is $k_0$-stationary if there 
exists a  holomorphic lift $\bm{f}=(f,\tilde{f})$ of $f$ to the cotangent bundle $T^*\C^{n+1}$, continuous up to 
the boundary and such that for all $\zeta \in b\Delta,\ \bm{f}(\zeta)\in\mathcal{N}^{k_0}S(\zeta)$
where
\begin{equation}\label{eqcon}
\mathcal{N}^{k_0}S(\zeta):=\{(z,w,\tilde{z},\tilde{w}) \in T^*\C^{n+1} \ | \ (z,w) \in S, (\tilde{z},\tilde{w}) \in 
\zeta^{k_0}N^*_z S\setminus \{0\} \},
\end{equation}
and where $N^*_z S=\spanc_{\R}\{\partial r(z)\}$ is the conormal fiber at $z$ of the hypersurface $S$. 
\end{defi}

Indeed, one can consider $k_0$-stationary discs as sections of the
circle bundle $\mathcal{N}^{k_0} S = \{(\zeta, \xi) \colon 
\xi\in \mathcal{N}^{k_0}S(\zeta)\} \subset S^1 \times \C^{2n+2}$. For 
a Levi-nondegenerate hypersurface $S$, this turns out to 
be totally real. 

We are interested in constructing $k_0$-stationary discs for Levi-degenerate hypersurfaces. Notice that in 
such a situation, the submanifold $\mathcal{N}^{k_0}S(\zeta)$ is not totally real for all $\zeta\in b\Delta$. 
In fact we are precisely interested in discs passing through the degeneracy locus of 
$\mathcal{N}^{k_0}S$. For this purpose, we will restrict our attention to discs which satisfy certain pointwise constraints.

\section{The model situation}\label{sec:models}

\subsection{Weighted polynomial models}\label{polmod}

A (real) polynomial $P:\C^n \to \C$ is {\it weighted homogeneous of weight $M=(m_1,\cdots,m_n) \in 
\N^n$ and (weighted) degree $d \in \N$} if for any real number $t$ and $z\in \C^n$  we have
$$P(t^{m_1}z_1,\cdots,t^{m_n}z_n,t^{m_1}\bar z_1,\cdots,t^{m_n}\bar z_n)=t^dP(z, \bar z).$$  
With the abbreviated notation $t^M z = (t^{m_1}z_1,\cdots,t^{m_n}z_n )$, the condition 
can be written as $P(t^M z, t^M \bar z) = t^d P(z, \bar z)$.
Note that a weighted homogeneous polynomial $P$ of weight $(1,\cdots,1)$ is homogeneous. 
%It is
% often convenient to assume that the $m_j$ do not have 
% any common divisor, but
 We shall  encounter circumstances in which it is more convenient to 
assume that   $m_1,\cdots,m_n$ are all even; since 
the actual size of the weights $(m_1, \dots , m_n)$ is often not so 
important, we shall assume most of the time that we work with such an ``even'' weight system. We notice that in the case of an even weight system all 
linear combinations of weights and also of possible homogeneities are even; in particular, 
 the numbers 
$d$ and  
$d-m_i$ and $d-m_i-m_j$ for $1\leq i,j \leq n$, are even.  

For two multi-indices $M=(m_1,\cdots,m_n)$ and $J=(j_1,\cdots,j_n)$ we write 
$$M\cdot J=\sum_{i=1}^n m_ij_i.$$   We now fix a weight (vector) $M=(m_1,\cdots,m_n)$ and
a real-valued, weighted homogeneous polynomial $P$ of (weighted) degree $d$, 
written as 
\begin{equation}\label{e:polydecomp} 
P(z, \bar z) =  \sum_{\substack{M\cdot J + M\cdot K = d \\ d-k_0\leq M\cdot J \leq k_0}}\alpha_{JK} z^J \overline z^K = \sum_{\ell= d-k_0 }^{k_0}  \underbrace{\left(\sum_{\substack{M\cdot J + M\cdot K = d \\ M\cdot K = \ell }}\alpha_{JK} z^J \overline z^K \right)}_{\normalsize:=P^{d - \ell, \ell} (z,\bar z)} 
\end{equation}
where $k_0$ is the largest $k$ with $\displaystyle \frac{d}{2}\leq k\leq d-1$ for which there exists two multi-indices $\tilde J, \tilde K$ 
with $M \cdot \tilde K=k$ satisfying $\alpha_{\tilde J \tilde K}\neq 0$. The $P^{d-\ell , \ell}$ are
the ``bihomogeneous'' components of $P$, satisfying $P^{d-\ell , \ell} (t^M z, s^M \bar z) = t^{d-\ell} s^\ell P^{d-\ell , \ell} (z, \bar z)$. Since $P$ is assumed to 
be real-valued, we have that   $\alpha_{JK} = \overline \alpha_{KJ}$ for all multi-indices $J, K$, and also, 
that $P^{d-\ell, \ell} (z,\bar z) = \bar{P}^{ \ell, d-\ell } (\bar z, z)$. 
We  define the {\em model hypersurface} $S_P=\{\rho=0\}\subset \mathbb C^{n+1}$ where 
\begin{equation}\label{eqmod}
\rho(z,w)=- \real w + P(z,\overline z)=- \real w +  \sum_{\substack{M\cdot J + M\cdot K = d \\ d-k_0\leq M\cdot J \leq k_0}}\alpha_{JK} z^J \overline z^K.
\end{equation}
     
 Define for $v = (v_1,\cdots,v_n) \in \C^n$ the analytic disc $h^v:\Delta \to \C^n$ 
\begin{equation*}
h^v(\zeta) = (1-\zeta)^M v =  ((1-\zeta)^{m_1}v_1,(1-\zeta)^{m_2}v_2, \ldots,(1-\zeta)^{m_n}v_n).
\end{equation*}
In analogy with the case of hypersurfaces in $\mathbb C^2$ \cite{be-de1}, we will need 
to control the Levi form of $S_P$ along the boundary of $h^v$, 
\[P_{z\overline z}(h^v(\zeta) ,\overline{h^v(\zeta)})=  \left(\begin{matrix}
P_{z_1 \overline z_1}(h^v(\zeta) ,\overline{h^v(\zeta)}) & \cdots &   P_{z_1 \overline z_n}(h^v(\zeta) ,\overline{h^v(\zeta)})     \\
\vdots & \ddots & \vdots \\
P_{z_n \overline z_1}(h^v(\zeta) ,\overline{h^v(\zeta)})  & \cdots & P_{z_n \overline z_n}(h^v(\zeta) ,\overline{h^v(\zeta)})    \\
\end{matrix}\right). \] 
For $\zeta\in b\Delta$ we have 
\begin{equation*}
\begin{aligned}
\zeta^{k_0}P_{z_i \overline z_j}(h^v(\zeta),\overline{h^v(\zeta)}) &= \sum_{\ell= d-k_0 }^{k_0} {P^{d - \ell, \ell}_{z_i \overline z_j} ((1-\zeta)^M v,(1- \bar \zeta)^M \bar v)} \\
&= \sum_{\ell= d-k_0 }^{k_0} (1- \zeta)^{d-\ell - m_i} (1 - \bar \zeta)^{\ell - m_j} \zeta^{k_0} {P^{d - \ell, \ell}_{z_i \overline z_j} ( v, \bar v)} \\
&= (1- \zeta)^{d-m_i - m_j} \sum_{\ell= d-k_0 }^{k_0} (-1)^{\ell- m_j} \zeta^{k_0 - \ell + m_j} {P^{d - \ell, \ell}_{z_i \overline z_j} ( v, \bar v)} \\
&= (1- \zeta)^{d-m_i - m_j} \zeta^{k_0} P_{z_i \overline z_j} (v, (-\bar \zeta)^M \bar v ). 
\end{aligned}
\end{equation*}
and with a similar computation for the $P_{z_i z_j}$ derivatives, we can thus write
\begin{equation*}
\left\{
\begin{array}{lll} 
\zeta^{k_0}P_{z_i \overline z_j}(h^v(\zeta),\overline{h^v(\zeta)}) = (1 - \zeta)^{d-m_i-m_j}Q^v_{i\overline j}(\zeta) \\
\\
\zeta^{k_0}P_{z_i  z_j}(h^v(\zeta),\overline{h^v(\zeta)}) = (1 - \zeta)^{d-m_i-m_j}S^v_{ij}(\zeta)\\
\end{array}\right.
\end{equation*}
where $Q^v_{i\overline j}$ and $S^v_{ij}$ are holomorphic polynomials, and where each $Q^v_{i\overline j}$ has  degree at most $2k_0 - d + m_j$ and each $S^v_{ij}$ has degree at most $2k_0 - d$. Furthermore, each $Q^v_{i\overline j}$ is divisible by $\zeta^{m_j}$; this observation will turn out to be crucial in the proof of our main result. Our crucial assumption is now that not only does $h^v$ only pass 
through Levi-nondegenerate points for $\zeta \neq 1$, but also, 
that the Levi form of $S_P$ along $h^v$ has the generic order of vanishing at $1$ 
(so that the order of vanishing of the Levi form is going to stay constant under 
small perturbations of both $P$ and $v$). To be exact:
\begin{defi}\label{defadm1}
We say that $v$ is {\em admissible for $P$} if there exists $g^v$ such that for
$f^v = (h^v, g^v)$ we have that $f^v (\partial\Delta) \subset S_P$, but   $f^v (\Delta ) \not\subset S_P$ and if 
 for  $\zeta \in b\Delta$
\begin{equation}\label{eqadm}
Q^v (\zeta)=\det \left(\begin{matrix}
Q_{1\overline 1}(\zeta) & \ldots & Q_{1\overline n}(\zeta)     \\
\vdots & \ddots & \vdots \\
Q_{n\overline 1}(\zeta)  & \ldots & Q_{n\overline n}(\zeta)    \\
\end{matrix}\right) \neq 0.
\end{equation}
\end{defi}
We also note that for a generic $P$, $Q(\zeta)$ has exactly degree $n(2k_0 - d) + \sum_{i=1}^n m_i$. Under 
generic conditions, we do find admissible vectors: 

\begin{lemma}
\label{lem:admissiblepoints} Assume that $S_P$ is generically Levi-nondegenerate,
and that the set of Levi-degenerate points $\Sigma_P = \{ (z,w) \in S_P \colon \det P_{z_i \bar z_j} (z,\bar z) = 0 
 \}$  does not have any branches of dimension $2n-1$ near $0$. Then there exists an admissible vector $v$ for $P$. 
\end{lemma}
\begin{proof}
We first claim that for an open, dense subset of $v$'s, 
we  have that their associated $Q^v$ vanishes only at $1$. Since 
\[ \begin{aligned} (1- \zeta)^{n d - 2 |M|} Q^v(\zeta) &= \zeta^{n k_0} \det \begin{pmatrix}
	P_{z_1 \bar z_1} (h^v (\zeta), \overline{ h^v (\zeta)}) & \dots & P_{z_1 \bar z_n} (h^v (\zeta), \overline{ h^v (\zeta)}) \\
	\vdots & & \vdots \\
		P_{z_n \bar z_1} (h^v (\zeta), \overline{ h^v (\zeta)}) & \dots & P_{z_n \bar z_n} (h^v (\zeta), \overline{ h^v (\zeta)})  
\end{pmatrix}\\ &=: \zeta^{n k_0} D^v(\zeta,\bar \zeta), \end{aligned}\] 
the zeroes of $Q^v$ for $\zeta\neq 1$ are exactly those points $\zeta\in \partial \Delta$ for which
 $(h^v (\zeta), \real P (h^v (\zeta) , \overline{h^v (\zeta)}) \in \Sigma_P $ is a Levi-degenerate point.  
Indeed, 
assume on the contrary that there exists an open set of 
$v$'s each of which has a $\zeta = \zeta_v$ with $D(\zeta_v) = 0$. Passing to a smooth point of the real 
algebraic variety $\Sigma_P$ we see that therefore its dimension would need to be at 
least $2n-1$, which is excluded by assumption. 

We next study the behaviour of $Q^v$ at $1$ and claim that 
for $v$ which satisfy that $D^v (1,1) \neq 0$ we have 
that $Q^v(1) \neq 0$. In order to see this, we replace the variable 
$\zeta \in \Delta$ with a variable $t$ in the upper half plane by the 
coordinate change $ \zeta = \frac{i-t}{i+t} $. We then have that
 $(1- \zeta) = 2 t + O(t^2)$, and the boundary $\partial \Delta$ corresponds to $\R$, so that 
 \[ \begin{aligned} D^v (\zeta, \bar \zeta) &= D^v (2 t + O(t^2) , 2 t + O(t^2) )  
 \\ &= (2t)^{nd-2|M|} D^v (1,1)  + O(t^{nd - 2|M|+1}). \end{aligned}\]
 It follows that $Q^v(1) = D^v (1,1) \neq 0$. The set of all $v$'s  for
 which $D^v (1,1) \neq 0$ is by assumption open and dense. 

 Lastly, we claim that 
 the set of vectors $v$ for which $h^v (\Delta ) \not \subset S_P $ is contained
 in the set of $v$'s  for which $P(v,\bar v) \neq 0$. The Lemma follows with that claim: Admissible vectors lie in the intersection of the three dense, open sets we have discussed. So assume that $f^v (\Delta ) \subset S_P$. Then it is easy to see that 
 $g^v (\zeta) = 0$ for $\zeta \in \Delta$. Hence $P((1-\zeta)^M v, (1 - \bar \zeta)^M \bar v) = 0$ throughout $\Delta$ and therefore $P(v,\bar v) = 0$.
\end{proof}

In particular, the disc $f^0$
\begin{equation}\label{eqinitf} 
f^0=(h^v,g^0)=((1-\zeta)^{m_1}v_1,\ldots, (1-\zeta)^{m_n}v_n),g^0)
\end{equation}
is a $k_0$-stationary disc attached to $S_P$ and satisfies $f^0(1)=0$.
We shall henceforth use $f^0$ to denote a (fixed) $k_0$ stationary disc 
associated with an admissible $v$.

\section{Construction of $k_0$-stationary discs}

In this section, we aim to construct $k_0$-stationary discs for suitable deformations of the model hypersurface studied in Section 
\ref{sec:models}. To this end, we first define a space $X$ parametrizing allowed deformations. 

\subsection{Space of allowed deformations}

Let $S_P=\{\rho=0\}$ be a weighted polynomial model of the form $(\ref{eqmod})$. Let $k>0$ be an integer.   
Choose $\delta > 0$ large enough so that $f^0\left(\overline \Delta\right)$, for $f^0$ defined in (\ref{eqinitf}), 
is contained in the polydisc ${\delta \Delta}^{n+1}\subset \mathbb C^{n+1}$.
Following \cite{be-de1}, we consider the affine Banach space $X$ of functions $r\in  \mathcal{C}^{k+3}\left(\overline {\delta\Delta
^{n+1}}\right)$  which can be written as
\begin{equation*}
r(z,w) = \rho(z,w) +\theta(z,\imag w )
\end{equation*}
with 
\begin{equation}\label{eqallow}
\theta(z,\imag w )= \sum_{M\cdot J+M\cdot K=d+1} (z^J \overline z^K) \cdot r_{JK0}(z)+
\sum_{l=1}^{d}\sum_{M\cdot J+M\cdot K=d-l} z^J \overline z^K (\imag w)^l \cdot r_{JKl}(z,\imag w)
\end{equation}
where $r_{JK0} \in  \mathcal C^{k+3}_\C\left(\overline{\delta\Delta^n}\right)$ and 
$r_{JKl}\in \mathcal  C^{k+3}_\C\left(\overline{\delta\Delta^n}\times[-\delta,\delta]\right)$. Furthermore, we equip $X$ with the 
following norm 
$$\vnorm{ r }_{X} = \sup \vnorm{ r_{JKl} }_{\mathcal C^{k+3}}$$
so that $X$ is isomorphic to a real closed subspace of a suitable power of 
$\mathcal  C^{k+3}_\C\left(\overline{\delta\Delta^n}\times [-\delta,\delta]\right)$ and, hence is a Banach space.
\begin{remark}
Equivalently, a defining function $r$ (of class $\mathcal{C}^{k+3}$) is 
an allowed deformation of $S_P$ if and only if 
\[ r_{z^J \bar z^K s^\ell} (0) = \begin{cases}
J! K! \alpha_{J,K} & M(J +K) = d, \, \ell = 0 \\
0 & M(J+K) + \ell < d. 
\end{cases}   \]
One can show that these conditions are independent of 
the choice of suitably adapted holomorphic coordinates (actually, they are independent 
with respect to CR diffeomorphisms of class $C^\ell$ whose linear parts preserve weights,  for $\ell$ large enough), 
and hence, that 
the definition of ``allowed deformation'' actually gives rise to 
a well-defined class of real hypersurfaces, independent 
of the coordinates used. 
\end{remark}

\subsection{Defining equations of $\mathcal{N}^{k_0}S$ and singular Riemann-Hilbert problems}
Let 
$$S_P = \{\rho=0\} = \{-\real w + P(z,\overline z) =0\}\subset \mathbb C^{n+1}$$ 
be a 
weighted model hypersurface  of the form (\ref{eqmod}). For $\zeta \in b\Delta$, the submanifold 
$\mathcal{N}^{k_0}S_P(\zeta)\subset \C^{2n+2}$ (see (\ref{eqcon})) may be defined by $2n+2$ explicit 
real equations. Indeed, we have
\begin{eqnarray*}
(z,w,\tilde{z},\tilde{w}) \in \mathcal{N}^{k_0}S_P(\zeta) &\Leftrightarrow&  
\left\{
\begin{array}{lll} 
\rho(z,w)=0 \\
\\
\mbox{there exists }\ c: b\Delta\rightarrow \R\setminus\{0\} \ \mbox{such that }\\

 (\tilde{z},\tilde{w})=\zeta^{k_0} c(\zeta)\left(P_z(z,\overline{z}),-\frac{1}{2}\right)

\end{array}
\right.\\
\\
&\Leftrightarrow  &
\left\{
\begin{array}{lll} 
\rho(z,w)=0 \\
\\
\displaystyle \frac{\tilde{w}}{\zeta^{k_0}} \in \R\\
\\
{\tilde{z}}_{i}+2\tilde{w}P_{z_i}(z,\overline{z})=0 \ \mbox{ for } 1\leq i\leq n.
\end{array}
\right.
\end{eqnarray*}\\
It follows that a set of $2n+2$ real defining equations for the submanifold $\mathcal{N}^{k_0}S_P(\zeta)\subset  \C^{2n+2}$ 
is given by
\begin{equation*}
\left\{
\begin{aligned} 
\tilde{\rho}_1(\zeta)(z,w,\tilde z, \tilde w) & =  - \real w + P(z,\overline z)=0\\
\tilde{\rho}_2(\zeta)(z,w,\tilde z, \tilde w) & =  \left({\tilde{z}}_{1}+2\tilde{w}P_{z_1}(z,\overline{z})\right) + 
\left(\overline{{\tilde{z}}_{1}+2\tilde{w}P_{z_1}(z,\overline{z})}\right) = 0\\
\tilde{\rho}_{3}(\zeta)(z,w,\tilde z, \tilde w) & =  i\left({\tilde{z}}_{1}+2\tilde{w}P_{z_1}(z,\overline{z})\right) - 
i\left(\overline{{\tilde{z}}_{1}+2\tilde{w}P_{z_1}(z,\overline{z})}\right) = 0\\
&\vdots & \\
\tilde{\rho}_{2n}(\zeta)(z,w,\tilde z, \tilde w) & =  \left({\tilde{z}}_{n}+2\tilde{w}P_{z_n}(z,\overline{z})\right) + 
\left(\overline{{\tilde{z}}_{n}+2\tilde{w}P_{z_n}(z,\overline{z})}\right) = 0\\
\tilde{\rho}_{2n+1}(\zeta)(z,w,\tilde z, \tilde w) & =  i\left({\tilde{z}}_{n}+2\tilde{w}P_{z_n}(z,\overline{z})\right) - 
i\left(\overline{{\tilde{z}}_{n}+2\tilde{w}P_{z_n}(z,\overline{z})}\right) = 0\\
\displaystyle \tilde{\rho}_{2n+2}(\zeta)(z,w,\tilde z, \tilde w) & =  \displaystyle i\frac{\tilde{w}}{\zeta^{k_0}}-i\zeta^{k_0}\overline{\tilde{w}} = 0.
\end{aligned}
\right.
\end{equation*}
% \begin{equation*}
% \left\{
% \begin{array}{lll} 

% \tilde{\rho}_1(\zeta)(z,w) & = & - \real w + P(z,\overline z)=0\\
% \\

% \tilde{\rho}_2(\zeta)(z,w) & = & \left({\tilde{z}}_{1}+2\tilde{w}P_{z_1}(z,\overline{z})\right) + 
% \left(\overline{{\tilde{z}}_{1}+2\tilde{w}P_{z_1}(z,\overline{z})}\right) = 0\\
% \\
% \tilde{\rho}_{3}(\zeta)(z,w) & = & i\left({\tilde{z}}_{1}+2\tilde{w}P_{z_1}(z,\overline{z})\right) - 
% i\left(\overline{{\tilde{z}}_{1}+2\tilde{w}P_{z_1}(z,\overline{z})}\right) = 0\\
% \\
% &\vdots & \\
% \\
% \tilde{\rho}_{2n}(\zeta)(z,w) & = & \left({\tilde{z}}_{n}+2\tilde{w}P_{z_n}(z,\overline{z})\right) + 
% \left(\overline{{\tilde{z}}_{n}+2\tilde{w}P_{z_n}(z,\overline{z})}\right) = 0\\
% \\
% \tilde{\rho}_{2n+1}(\zeta)(z,w) & = & i\left({\tilde{z}}_{n}+2\tilde{w}P_{z_n}(z,\overline{z})\right) - 
% i\left(\overline{{\tilde{z}}_{n}+2\tilde{w}P_{z_n}(z,\overline{z})}\right) = 0\\
% \\
% \displaystyle \tilde{\rho}_{2n+2}(\zeta)(z,w) & = & \displaystyle i\frac{\tilde{w}}{\zeta^{k_0}}-i\zeta^{k_0}\overline{\tilde{w}} = 0.\\
% \end{array}
% \right.
% \end{equation*}
We set 
$$\tilde{\rho}:=({\tilde{\rho}}_{1},\cdots,\tilde{\rho}_{2n+2}).$$ 
For a general hypersurface $S=\{r=0\}$ with $r \in X$ in the space of allowed deformations,  we denote by 
$\tilde{r}(\zeta)$ the corresponding defining functions of $\mathcal{N}^{k_0}S(\zeta)$. This allows to consider lifts of stationary discs as solutions of a nonlinear 
Riemann-Hilbert type problem with singularities. More precisely, a holomorphic disc $\bm{f} \in \left(\mathcal{A}^{k,\alpha}\right)^{2n+2}$ is the lift of a $k_0$-stationary disc attached to $S$ 
if and only if 
\begin{equation}\label{eqstatrh}
\tilde{r}(\bm{f})=0 \ \mbox{ on } b\Delta.
\end{equation}
The next  section is devoted to the  study of the nonlinear problem  (\ref{eqstatrh}). Its linearization leads to a singular linear Riemann-Hilbert problem which can be treated with the 
techniques developed in \cite{be-de2}.

\subsection{Construction of $k_0$-stationary discs}\label{rhp}

Let 
$$S_P = \{\rho=0\} = \{-\real w + P(z,\overline z)=0 \}\subset \mathbb C^{n+1}$$ 
be a 
weighted model hypersurface  of the form (\ref{eqmod}) with weight $M=(m_1,\cdots,m_n)$ and degree $d$. Let $v=(v_1,\cdots,v_n)$ be an admissible vector for 
$P$. Consider a real hypersurface $S=\{r=0\}$ with $r \in X$. 
We introduce the following space of maps
\begin{equation}\label{eqdefY2}
Y^{M,d}:=\prod_{i=1}^n \left(\mathcal A^{k,\alpha}_{0^{m_i}}\right) \times \mathcal A^{k,\alpha}_{0}  \times \prod_{i=1}^n \left(\mathcal A^{k,\alpha}_{0^{d-m_i}} \right)
\times \mathcal A^{k,\alpha}
\end{equation}
endowed with the product norm defined in Equation (\ref{eqnorm}). 
We denote by $\mathcal S^{k_0,r}$ the set of lifts $\bm{f} \in Y^{M,d}$ of $k_0$-stationary discs 
for the hypersurface $S=\{r=0\}$. 
Following Section \ref{sec:models}, we consider the initial $k_0$-stationary disc attached to $S_P$ given by 
$$\bm{f^0}=(h^0,g^0,\tilde{h}^0,\tilde{g}^0)=((1-\zeta)^{m_1}v_1,\cdots,(1-\zeta)^{m_n}v_n,g^0,\tilde{h}^0,-\zeta^{k_0}/2) \in Y^{M,d}$$
where $\tilde{h}^0(\zeta)=\zeta^{k_0}P_z(h^0, \overline{h^0})$. 
We have: 
\begin{theorem}\label{theodiscs2}
Under the above assumptions, there exist an integer $N$,  open 
neighborhoods $V$ of $\rho$ in $X$  and $U$ of $0$  in $\R^{N}$, a real number $\eta>0$ and a map
$$\mathcal{F}:V \times U \to  Y^{M,d}$$
 of class $\mathcal{C}^1$ such that:
\begin{enumerate}[i.]
\item $\mathcal{F}(\rho,0)=\bm{f^0}$,

\item for all $r\in V$ the map 
$$\mathcal{F}(r,\cdot):U\to \{\bm{f} \in \mathcal{S}^{k_0,r}\ \ | \  
\|\bm{f}-\bm{f^0}\|_{Y^{M,d}}<\eta\}$$
is one-to-one and onto.
\end{enumerate}
\end{theorem}
\begin{remark} 
In the proof of Theorem \ref{theodiscs2}, we show that the dimension $N$ is estimated above by $2(n+1)(k_0+1)+2nk_0-2dn$. Since this dimension depends on 
the choice of the weights $(m_1,\cdots,m_n)$, a precise computation of $N$ is not relevant for our approach.
\end{remark}
\begin{proof}
In a neighborhood of $(\rho,\bm{f^0})$ in $X\times Y^{M,d}$, we define the following map between Banach spaces
\[\mathcal H: X\times Y^{M,d}
\to \mathcal C_0^{k,\alpha}  \times \prod_{i=1}^n\left(\left(\mathcal C_{0^{d-m_i}}^{k,\alpha}\right)^2\right) \times  \mathcal C^{k,\alpha}\]
by 
$$
\mathcal H(r,\bm{f}):=\tilde{r}(\bm{f}).
$$
Here we use the notation 
$$\tilde{r}(\bm{f})(\zeta)=\tilde{r}(\zeta)(\bm{f}(\zeta)).$$ It follows from the definition of the Banach spaces $X$ and $Y^{M,d}$
 that the map $\mathcal H$ is of class $\mathcal{C}^1$; the proof of this claim is  analogous to the proof of Lemma 3.3 in \cite{be-de1} 
 (see also Lemma 5.1 in \cite{hi-ta} and Lemma 11.1 in \cite{gl1}). Recall that a holomorphic disc $\bm{f} \in Y^{M,d}$ is the lift of a $k_0$-stationary disc attached to 
 $S=\{r=0\}$ if and only if it solves the nonlinear Riemann-Hilbert problem (\ref{eqstatrh}). In other words, for any fixed $r\in X$, 
 the zero set of $\mathcal H(\tilde{r},\cdot)$ coincides with 
$\mathcal{S}^{k_0,r}$. In order to show Theorem \ref{theodiscs2} we apply the implicit function theorem to the map $\mathcal{H}$. To this end, 
we need to  consider the 
partial derivative of $\mathcal H$ with respect to $Y^{M,d}$ at $(\rho,\bm{f^0})$
\begin{equation}\label{eqrh0}
\bm{f}'\mapsto 2\real \left[\overline{G(\zeta)}\bm{f}'\right]
\end{equation}
where the matrix 
$$G(\zeta):=\left({\tilde\rho}_{\overline{z}}(\bm{f^0}),
{\tilde\rho}_{\overline{w}}(\bm{f^0}),
{\tilde\rho}_{\overline{\tilde{z}}}(\bm{f^0}),{\tilde\rho}_{\overline{\tilde{w}}}(\bm{f^0})\right) \in M_{2n+2}(\C)$$ has the following expression

$$G(\zeta)=\begin{pmatrix}

P_{{\overline{z}_1}}(h^0,\overline{h^0}) & \hdots & P_{{\overline{z}_n}}(h^0,\overline{h^0})  & -1/2  &  0 &  \cdots & 0 &  0  \\

 &  & &  0  & 1  & \ddots & 0 & 2\overline{P_{z_1}(h^0,\overline{h^0})} \\

     &  & &    0 &  -i& \ddots & 0  & -2i\overline{P_{z_1}(h^0,\overline{h^0})} \\
  &  B(\zeta) & &    0  &  0 & \ddots & 0  &  2\overline{P_{z_2}(h^0,\overline{h^0})}  \\

  &  & &  \vdots & \vdots & \ddots & \vdots & \vdots \\

     &  & &    0  & 0  & \ddots &  -i  & -2i\overline{P_{z_n}(h^0,\overline{h^0})}  \\
0  & \cdots & 0 &   0  &   0  &  \ddots & 0 & -i\zeta^{k_0} \\

\end{pmatrix}.$$
Using the notation $d_{\ell j}:=d-m_\ell-m_j$, the entries of the  $2n \times n$  matrix $B(\zeta)$ are given by  
$$B_{2\ell-1,j}(\zeta)=-(1 - \zeta)^{d_{\ell j}}\left(Q_{\ell \overline j}(\zeta) +\frac{\overline{S}_{\ell j}(\zeta)}{\zeta^{d_{\ell j}}}\right)$$
for odd $1 \leq 2l-1 \leq 2n-1$  and   
$$B_{2\ell,j}(\zeta)=-i(1 - \zeta)^{d_{\ell j}}\left(Q_{\ell\overline j}(\zeta) -\frac{\overline{S}_{\ell j}(\zeta)}{\zeta^{d_{\ell j}}}\right)$$
 for even $2 \leq 2\ell \leq 2n$.

In order to apply the implicit function theorem, we need to study the kernel and surjectivity of  the map
$\bm{f}'\mapsto 2\real \left[\overline{G(\zeta)}\bm{f}'\right]$. After permuting columns of $G(\zeta)$, we consider the  following operator
$$L_1:\mathcal A^{k,\alpha}_{0} \times \prod_{i=1}^n \left(\mathcal A^{k,\alpha}_{0^{d-m_i}} \times \mathcal A^{k,\alpha}_{0^{m_i}}\right) \times \mathcal A^{k,\alpha}  \to  \mathcal C_0^{k,\alpha}  \times \prod_{i=1}^n\left(\left(\mathcal C_{0^{d-m_i}}^{k,\alpha}\right)^2\right) \times  \mathcal C^{k,\alpha}$$
given by 
\begin{equation*}
L_1(g',{\tilde{h}}_{1}',h_1',\cdots,{\tilde{h}}_{n}', h_n',\tilde{g}'):=2\real \left[\overline{G_1(\zeta)}(g',{\tilde{h}}_{1}',h_1',\cdots,{\tilde{h}}_{n}',h_n', \tilde{g}')\right],
\end{equation*}
where
\[G_1(\zeta)= \left(\begin{array}{cccccccccccc}
-1/2 & & (*)\\
 & A(\zeta) &  \\
 (0) &  &  -i\zeta^{k_0}\\
\end{array}\right)\] 
and where $A(\zeta)$ is  
$$\left(\begin{matrix}
1&B_{1,1}(\zeta) &   \hdots& 0 &B_{1,n}(\zeta)         \\

-i & B_{2,1}(\zeta)   &  \hdots & 0  &B_{2,1}(\zeta)       \\

\vdots & \vdots & \vdots & \vdots & \vdots  \\

0& B_{2n-1,1}(\zeta)  &   \hdots & 1 & B_{2n-1,n}(\zeta)      \\

0 &B_{2n,1}(\zeta)   &  \hdots  &   -i   & B_{2n,n}(\zeta)    \\

\end{matrix}\right).$$
The kernels of the differential map (\ref{eqrh0}) and  $L_1$ are of the same dimension, and 
the map $(\ref{eqrh0})$ is onto if and only if $L_1$ is onto.

Note that since $G_1(1)$ is not invertible, the classical techniques developed in \cite{fo, gl1, gl2} to study the corresponding linear Riemann-Hilbert problem cannot be directly applied. Therefore the following step is crucial in our approach since it  allows one to reduce a linear singular 
Riemann-Hilbert problem to a 
 regular one with homogeneous pointwise constraints, and allows then the use of Theorem 2.1 in \cite{be-de2}.  For $\varphi \in  \mathcal C_0^{k,\alpha}  \times \prod_{i=1}^n\left(\left(\mathcal C_{0^{d-m_i}}^{k,\alpha}\right)^2\right) \times  \mathcal C^{k,\alpha}$, we manipulate the linear system 
$$2\real \left[\overline{G_1(\zeta)}(g',{\tilde{h}}_{1}',h_1',\cdots,{\tilde{h}}_{n}',h_n', \tilde{g}')\right]=\varphi $$ in the following way. 
We divide the first line by $(1-\zeta)$ and the $(2\ell-1)^{th}$ and $(2\ell)^{th}$ lines by  $(1 - \zeta)^{d-m_\ell}$, for $l=1,\cdots,n.$ Following Lemma \ref{subspaces}, we then multiply 
the $(2\ell-1)^{th}$ and $(2\ell)^{th}$ lines by $\zeta^{s_\ell}$, where $\displaystyle s_\ell:=\frac{d-m_\ell}{2}$, $\ell=1,\cdots,n$. The resulting linear operator
\begin{equation}\label{eqL2}
L_2: (\mathcal A^{k,\alpha})^{2n+2} \to \mathcal R_1 \times (\mathcal R_0)^{2n} \times  \mathcal C^{k,\alpha}
\end{equation}
is equivalent to $L_1$ with respect to the properties we are interested in, namely its surjectivity and the description of its kernel. The new linear operator $L_2$, and its corresponding matrix $G_2$,  are of the form considered in Theorem 2.1 and Theorem 2.2 \cite{be-de2}. 
We have thus reduced the problem  to studying the linear operator 
$$L_3:(\mathcal A^{k,\alpha})^{2n} \to (\mathcal R_0)^{2n} $$
 defined by 
 $$L_3({\tilde{h}}_{1}',h_1',\cdots,{\tilde{h}}_{n}', h_n'):=2\real \left[\overline{A(\zeta)}({\tilde{h}}_{1}',-h_1',\cdots,{\tilde{h}}_{n}',-h_n')\right]$$   
where the corresponding matrix, still denoted by $A(\zeta)$, is
\begin{equation*}
\begin{pmatrix}

\overline \zeta^{s_1} & Q_{1\overline 1}\zeta^{s_1-m_1} + \overline S_{11}\overline\zeta^{s_1}  &  \ldots & 0 & Q_{1\overline n}\zeta^{s_1-m_n} + \overline S_{1n}\overline\zeta^{s_1} \\
  -i\overline \zeta^{s_1} &   iQ_{1\overline 1}\zeta^{s_1-m_1} - i  \overline S_{11}\overline\zeta^{s_1}   & \ldots & 0 & iQ_{1\overline n}\zeta^{s_1-m_n} - i\overline S_{1n}\overline\zeta^{s_n} \\

\vdots & \vdots & \vdots & \vdots & \vdots & \\

  0   &   Q_{n\overline 1}\zeta^{s_n-m_1} + \overline S_{n1}\overline\zeta^{s_n}  &  \ldots &  \overline\zeta^{s_n}   & Q_{n\overline n}\zeta^{s_n-m_n} + \overline S_{nn}\overline\zeta^{s_n} \\

0  &   iQ_{n\overline 1}\zeta^{s_n-m_1} -i \overline S_{n1}\overline\zeta^{s_n}  &   \ldots & -i\overline \zeta^{s_n}  &  iQ_{n\overline n}\zeta^{s_n-m_n} - i\overline S_{nn}\overline\zeta^{s_n}  \\

\end{pmatrix}.
\end{equation*}

Out of convenience, we set $Q'_{\ell\overline j}=Q_{\ell\overline j}\zeta^{-m_j}$ and therefore
\begin{equation*}
A(\zeta)=\left(\begin{matrix}

\overline \zeta^{s_1} & Q'_{1\overline 1}\zeta^{s_1} + \overline S_{11}\overline\zeta^{s_1}  & \ldots & 0 & Q'_{1\overline n}\zeta^{s_1} + \overline S_{1n}\overline\zeta^{s_1} \\

  -i\overline \zeta^{s_1} &   iQ'_{1\overline 1}\zeta^{s_1} - i  \overline S_{11}\overline\zeta^{s_1}  &  \ldots & 0 & iQ'_{1\overline n}\zeta^{s_1} - i\overline S_{1n}\overline\zeta^{s_n} \\

\vdots & \vdots & \vdots & \vdots & \vdots & \\

  0   &   Q'_{n\overline 1}\zeta^{s_n} + \overline S_{n1}\overline\zeta^{s_n}  &  \ldots &  \overline\zeta^{s_n}   & Q'_{n\overline n}\zeta^{s_n} + \overline S_{nn}\overline\zeta^{s_n} \\

0  &   iQ'_{n\overline 1}\zeta^{s_n} -i \overline S_{n1}\overline\zeta^{s_n}  &    \ldots & -i\overline \zeta^{s_n}  &  iQ'_{n\overline n}\zeta^{s_n} - i\overline S_{nn}\overline\zeta^{s_n}  \\

\end{matrix}\right).
\end{equation*}
Note that by manipulating rows of $A$ one shows that   
\begin{equation}\label{eqdet}
\det A(\zeta)=(2i)^nQ'(\zeta)
\end{equation}
where 
$$Q'(\zeta)=\zeta^{-(m_1+\cdots+m_n)}Q(\zeta).$$
\begin{lemma}\label{lemsurj} 
The linear operator $L_3 :(\mathcal A^{k,\alpha})^{2n} \to (\mathcal R_0)^{2n}$ is onto.
\end{lemma}
\begin{proof}[Proof of Lemma \ref{lemsurj}] 
According to Theorem 2.1 in \cite{be-de2} (with $m=0$), we need to show that the partial indices of the matrix
\[\overline {A^{-1}(\zeta)} A(\zeta) = \frac{1}{\overline {\det A (\zeta)}} A'(\zeta) = \frac{1}{\overline {(2i)^nQ'(\zeta)}} A'(\zeta)\]
are greater than or equal to $-1$. For $1\leq j,\ell\leq 2n$ we denote by $A'_{j\ell}$ the $(j,\ell)$-entry of $A'$.
 A direct computation gives for $\ell, p=1,\cdots,n$
$$
\begin{aligned}
A'_{2{\ell-1},2p}& =  (-2i)^n\zeta^{s_1+\cdots+s_n-s_\ell} \det \left(\begin{matrix}

 Q'_{l\overline p}\zeta^{s_\ell}&  S_{l1} \zeta^{s_\ell}&  S_{l2}\zeta^{s_\ell}  &  \cdots &  S_{l n}\zeta^{s_\ell}  \\

   \overline S_{1p}\overline{\zeta}^{s_1}&  \overline Q'_{1\overline 1} \overline{\zeta}^{s_1}& \overline Q'_{1\overline 2} \overline{\zeta}^{s_1} &  \cdots &   \overline Q'_{1\overline{n}} \overline{\zeta}^{s_1}\\

 \overline S_{2p} \overline{\zeta}^{s_2}&  \overline Q'_{2\overline 1}\overline{\zeta}^{s_2} &  \overline Q'_{2\overline 2}\overline{\zeta}^{s_2} &  \cdots &  \overline Q'_{2\overline{n}}\overline{\zeta}^{s_2}  \\
  
   \vdots  &   \vdots  &  \vdots  &   \vdots   \cdots &   \vdots   \\

 \overline S_{n p} \overline{\zeta}^{s_n}&  \overline Q'_{n \overline 1} \overline{\zeta}^{s_n}&  \overline Q'_{n \overline 2} \overline{\zeta}^{s_n}&  \cdots &  \overline Q'_{n \overline{n}}\overline{\zeta}^{s_n}  \\

\end{matrix}\right)\\
\\
& =  (-2i)^n\det  \underbrace{\left(\begin{matrix}
 Q'_{\ell\overline p}&  S_{\ell1} &  S_{\ell2}  &  \cdots &  S_{\ell n} \\
   \overline S_{1p}&  \overline Q'_{1\overline 1} & \overline Q'_{1\overline 2}  &  \cdots &   \overline Q'_{1\overline{n}} \\
 \overline S_{2p}&  \overline Q'_{2\overline 1} &  \overline Q'_{2\overline 2} &  \cdots &  \overline Q'_{2\overline{n}} \\
   \vdots  &   \vdots  &  \vdots  &   \vdots   \cdots &   \vdots   \\
 \overline S_{n p} &  \overline Q'_{n \overline 1} &  \overline Q'_{n \overline 2} &  \cdots &  \overline Q'_{n \overline{n}}  \\
\end{matrix}\right)}_{\normalsize:=B_{2\ell-1,2p}}\\
\\
&=(-2i)^n a'_{2{\ell-1},2p}\\
\end{aligned}
$$
where $a'_{2{\ell-1},2p}=\det B_{2\ell-1,2p}$.
For a square matrix $B$, we write $C_{j\ell}(B)$ for its $(j,\ell)$-cofactor.  Notice that for all $j=1,\cdots,n$ and  any $p,p'=1,\cdots,n$, we have 
$$C_{j,1}(B_{2{\ell-1},2p})=C_{j1}(B_{2{\ell-1},2p'}).$$  
We denote this cofactor by $C_{j,1;\ell}$. We also have 
$$C_{1,j}(B_{2{\ell-1},2p})=C_{1,j}(B_{2{\ell'-1},2p})$$
for any $p= 1,\dots ,n$ and every $\ell,\ell'= 1,\dots, n$
which will be denoted by  $C_{1,j}^p$.
A straightforward computation leads to
$$A'_{2{\ell-1},2p-1}=(-2i)^n C_{p+1,1;\ell}$$
and 
$$A'_{2\ell,2p}=(-2i)^n C_{1,\ell+1}^p$$
for $\ell, p=1,\cdots,n$.
 Denote by $D_{\ell p}$ the $n\times n$  matrix obtained by removing the first row and $(\ell+1)^{\rm th}$ column of $B_{2{\ell-1},2p}$, namely
$$D_{\ell p}=\left(\begin{matrix}
   \overline S_{1p}&  \overline Q'_{1\overline 1} & \overline Q'_{1\overline 2}   &  \cdots &   \overline Q'_{1\overline {\ell-1}}  & \overline Q'_{1\overline{\ell+1}}   &  \cdots &   \overline Q'_{1\overline{n}}  \\

 \overline S_{2p} &  \overline Q'_{2\overline 1}  &  \overline Q'_{2\overline 2} &  \cdots &   \overline Q'_{2\overline {\ell-1}} & \overline Q'_{2\overline {\ell+1}}  &  \cdots &  \overline Q'_{2\overline{n}}  \\
   \vdots  &   \vdots  &  \vdots  &   \cdots & \vdots &   \vdots &  \cdots & \vdots   \\

 \overline S_{n p}&  \overline Q'_{n \overline 1} &  \overline Q'_{n \overline 2} &  \cdots &   \overline Q'_{n\overline {\ell-1}} & \overline Q'_{n\overline{\ell+1}}  &  \cdots&  \overline Q'_{n \overline{n}}  \\
\end{matrix}\right)$$
for $\ell, p=1,\cdots,n$. Note that 
$$\det \left(D_{\ell p}\right)=(-1)^{\ell} C_{1,\ell+1}^p$$ 
and  
$$C_{j,1}(D_{\ell p})=C_{j,1}(D_{\ell p'})$$
which we denote by $c_{j,1;l}$. 
A direct computation gives 
$$A'_{2\ell,2p-1}=(-1)^{\ell+1}(-2i)^n c_{p,1;\ell}.$$
Therefore
\begin{equation}\label{eqa'bef}
\frac{A'(\zeta)}{(-2i)^n}=\begin{pmatrix}

C_{2,1;1} &  a'_{1,2} &  C_{3,1;1}  &  a'_{1,4}&  \cdots & C_{n+1,1;1}   &  a'_{1,2n} \\

c_{1,1;1}&   C_{1,2}^1  &  c_{2,1;1}  &  C_{1,2}^2  &  \cdots & c_{n,1;1} & C_{1,2}^n   \\
 
C_{2,1;2} &  a'_{3,2} &  C_{3,1;2} &  a'_{3,4} &  \cdots & C_{n+1,1,2}  &  a'_{3,2n}\\

-c_{1,1;2} &  C_{1,3}^1   &-c_{2,1;2}    &  C_{1,3}^2  &  \cdots & -c_{n,1;2} & C_{1,3}^n  \\

\vdots  &  \vdots&   \vdots &  \vdots&  \cdots & \vdots &  \vdots \\

C_{2,1;n}&  a'_{2n-1,2} &   C_{3,1;n}  &  a'_{2n-1,4} &  \cdots & C_{n+1,1;n} &  a'_{2n-1,2n} \\
 
 \frac{ c_{11,n}}{(-1)^{n+1}} &  C_{1,n+1}^1   &      \frac{c_{2,1;n}}{(-1)^{n+1}}  &  C_{1,n+1}^2  &  \cdots &     \frac{c_{n,1;n}}{(-1)^{n+1}} & C_{1, n+1}^n  \\

\end{pmatrix}
.
\end{equation}
Denote by $C_p$ the $p^{\rm th}$ column of $A'(\zeta)$.   Notice that performing the following column operation 
\begin{equation}\label{eqop}
C_{2p} \to C_{2p}-\sum_{j=1}^{n} \overline{S}_{jp}C_{2j-1}
\end{equation} 
for each $p=1,\cdots,n$
transforms $A'(\zeta)$ into

\begin{equation}\label{eqa'aft}
 A'(\zeta) = (-2i)^n\left(\begin{matrix}

C_{2,1;1} &  Q'_{1\overline 1} \overline{Q'}&   C_{3,1;1}  & Q'_{1\overline 2} \overline{Q'} &  \cdots & C_{n+1,1;1}  &  Q'_{l\overline n} \overline{Q'} \\

c_{1,1;2} &   0  &  c_{2,1;2}  &  0  &  \cdots & c_{n,1;2} & 0   \\

C_{2,1;2} & Q'_{2\overline 1} \overline{Q'} &  C_{3,1;2}  &  Q'_{2\overline 2} \overline{Q'} &  \cdots & C_{n+1,1;2}  & Q'_{2\overline n} \overline{Q'}\\

-c_{1,1;3} &  0   &-c_{2,1;3}    &  0  &  \cdots & -c_{n,1;3} & 0  \\

\vdots  &  \vdots&   \vdots &  \vdots&  \cdots & \vdots &  \vdots \\

C_{2,1;n} & Q'_{n\overline 1} \overline{Q'} &   C_{3,1;n}  & Q'_{n\overline 2} \overline{Q'} &  \cdots & C_{n+1,1;n}  & Q'_{n\overline n} \overline{Q'} \\
 
  c\frac{_{1,1;n} }{(-1)^{n+1}}& 0  &      \frac{c_{2,1;n}}{(-1)^{n+1}}  &  0  &  \cdots &     \frac{c_{n,1;n}}{(-1)^{n+1}} & 0  \\

\end{matrix}\right).
\end{equation} 

\

Now let $\kappa_1 \geq \ldots \geq \kappa_{2n}$ be the partial indices of $\overline {A^{-1}} A$, and let $\Lambda$ be the diagonal matrix with entries $\zeta^{\kappa_1}, \ldots, \zeta^{\kappa_{2n}}$. According to Lemma 5.1 in \cite{gl1} there exists a smooth map $\Theta: \overline{\Delta} \to \GL_{2n}(\C)$, holomorphic on $\Delta$, such that
\begin{equation}\label{eqtheta}
\Theta \overline {A^{-1}} A = \Lambda \overline \Theta.
\end{equation}
Denote by $\lambda = (\lambda_1,\mu_1,\ldots, \lambda_n, \mu_n)$ the last row of the matrix $\Theta$.
Using (\ref{eqa'bef}) and (\ref{eqtheta}), we get the following system:
 \begin{equation*}
 \left.\begin{aligned}
 \sum_{k=1}^n C_{j+1,1;k}\lambda_k   + \sum_{k=1}^n(-1)^{k+1} c_{j,1;k+1} \mu_k    &  =  \overline{ Q'} \zeta^{\kappa_{2n}} \overline \lambda_j \\ 
 \sum_{k=1}^n a'_{2k-1,2j}\lambda_k + \sum_{k=1}^n C^j_{1 ,k+1}\mu_k  &=  \overline Q' \zeta^{\kappa_{2n}} \overline \mu_j \\
 \end{aligned}\right\} j=1,\dots, n
 \end{equation*}

%  \begin{equation*}
% \left\{
% \begin{array}{lll} 

% C_{21,1}\lambda_1   + c_{11,2}\mu_1  +C_{21,2}\lambda_2 -  c_{11,3}\mu_2 + \ldots  +C_{21,n} \lambda_n +  (-1)^{n+1} c_{11,n} \mu_{n}  &  = &  \overline Q' \zeta^{\kappa_{2n}} \overline \lambda_1 \\ 
% \\
%  a'_{12}\lambda_1 + C^1_{12}\mu_2 + a'_{32}\lambda_2 + C_{13}^1\mu_3 + \ldots  +a'_{2n-1 \ 2} \lambda_n +  C_{1 \ n+1}^1 \mu_{n+1} &= &  \overline Q' \zeta^{\kappa_{2n}} \overline \mu_1 \\
% \\
% C_{31,1}\lambda_1   +c_{21,2}\mu_1  +C_{31,2}\lambda_2 -   c_{21,3}\mu_2 + \ldots  +C_{31,n} \lambda_n +  (-1)^{n+1} c_{21,n} \mu_{n}  & = &  \overline Q' \zeta^{\kappa_{2n}} \overline \lambda_2 \\ 
% \\
%  a'_{14}\lambda_1 + C_{12}^2\mu_1 + a'_{34}\lambda_2 + C_{13}^2\mu_2 + \ldots  +a'_{2n-1 \ 4} \lambda_n +  C_{1 \ n+1}^2 \mu_{n} & =&   \overline Q' \zeta^{\kappa_{2n}} \overline \mu_2 \\
 
%  \vdots & \vdots & \vdots \\
 
%  C_{n+1\ 1,1}\lambda_1   +c_{n1,2}\mu_1  +C_{n+1 \ 1,2}\lambda_2 -   c_{n1,3}\mu_2 + \ldots  +C_{n+1 \ 1,n} \lambda_n +  (-1)^n c_{n1,n} \mu_{n} &  =&   \overline Q' \zeta^{\kappa_{2n}} \overline \lambda_n \\ 
% \\
%  a'_{1 \ 2n}\lambda_1 + C_{12}^n\mu_1 + a'_{3\ 2n}\lambda_2 + C_{13}^n\mu_2 + \ldots  +a'_{2n-1 \ 2n} \lambda_n +  C_{1 \ n+1}^n \mu_{n} & = &  \overline Q' \zeta^{\kappa_{2n}} \overline \mu_n.\\
%  \end{array}
% \right.
% \end{equation*}
Performing operations (\ref{eqop}), and considering only the lines of the system coming from the 
second line above, we obtain the following (see (\ref{eqa'aft})):
\begin{equation*}
\begin{aligned}
\overline{ Q' }\sum_{k=1}^n Q'_{k \overline j}  \lambda_k    & = \overline{ Q' } \zeta^{\kappa_{2n}} \overline \mu_j - \sum_{k=1}^{n} \overline{S}_{k j}\overline{ Q' }\zeta^{\kappa_{2n}} \overline \lambda_k, \quad j = 1, \dots, n
\end{aligned}
\end{equation*}

%  \begin{equation*}
% \left\{
% \begin{array}{lll} 
% Q'_{1\overline 1} \overline Q' \lambda_1    +Q'_{2\overline 1} \overline Q' \lambda_2 + \ldots  +Q'_{n\overline 1} \overline Q' \lambda_n & =& \overline Q' \zeta^{\kappa_{2n}} \overline \mu_1 - \sum_{j=1}^{n} \overline{S}_{j1}\overline Q' \zeta^{\kappa_{2n}} \overline \lambda_j\\ 
% \\
% Q'_{1\overline 2} \overline Q' \lambda_1    +Q'_{2\overline 2} \overline Q' \lambda_2 + \ldots  +Q'_{n\overline 2} \overline Q' \lambda_n & = & \overline Q' \zeta^{\kappa_{2n}} \overline \mu_2 - \sum_{j=1}^{n} \overline{S}_{j1}\overline Q' \zeta^{\kappa_{2n}} \overline \lambda_j\\
%  \vdots & \vdots & \vdots \\
 
% Q'_{1\overline n} \overline Q' \lambda_1    +Q'_{2\overline n} \overline Q' \lambda_2 + \ldots  +Q'_{n\overline n} \overline Q' \lambda_n&  =&   \overline Q' \zeta^{\kappa_{2n}} \overline \mu_n - \sum_{j=1}^{n} \overline{S}_{j1}\overline Q' \zeta^{\kappa_{2n}} \overline \lambda_j.\\ 
%  \end{array}
% \right.
% \end{equation*}
Dividing by $\overline{Q'}$, which by assumption (\ref{eqadm}) is non-vanishing for all $\zeta\in b\Delta$, we have
 \begin{equation*}
\sum_{k=1}^n Q'_{k\overline j} \lambda_k  =   \zeta^{\kappa_{2n}} \overline \mu_j + \sum_{k=1}^{n} \overline{S}_{k j} \zeta^{\kappa_{2n}} \overline \lambda_k, \quad j = 1, \dots, n.
\end{equation*}

%  \begin{equation*}
% \left\{
% \begin{array}{lll} 

% Q'_{1\overline 1} l_1    +Q'_{2\overline 1} l_2 + \ldots  +Q'_{n\overline 1} l_n & =&   \zeta^{\kappa_{2n}} \overline q_1 + \sum_{j=1}^{n} \overline{S}_{j1} \zeta^{\kappa_{2n}} \overline l_j\\ 
% \\
% Q'_{1\overline 2}  l_1    +Q'_{2\overline 2} l_2 + \ldots  +Q'_{n\overline 2} l_n & =& \zeta^{\kappa_{2n}} \overline q_2 + \sum_{j=1}^{n} \overline{S}_{j1} \zeta^{\kappa_{2n}} \overline l_j\\
%  \vdots & \vdots & \vdots \\

% Q'_{1\overline n}  l_1    +Q'_{2\overline n} l_2 + \ldots  +Q'_{n\overline n}  l_n & =&   \zeta^{\kappa_{2n}} \overline q_n +\sum_{j=1}^{n} \overline{S}_{j1} \zeta^{\kappa_{2n}} \overline l_j.\\ 
% \end{array}
% \right.
% \end{equation*}
Recall that $Q_{i\overline j}$ is divisible by $\zeta^{m_j}$ (see Section \ref{polmod}) and thus $Q'_{i\overline j}=Q_{i\overline j}\zeta^{-m_j}$ is holomorphic. Now if 
$\kappa_{2n}\leq -1$, the right hand side of each one of the equations above is antiholomorphic (and divisible by $\overline \zeta$), while the left hand side is holomorphic. Thus they must both vanish, leading to the system
\[ \sum_{k=1}^n Q'_{k\overline j} \lambda_k = 0 , \quad j = 1, \dots , n.  \]
%  \begin{equation*}
% \left\{
% \begin{array}{lll} 

% Q'_{1\overline 1} l_1    +Q'_{2\overline 1} l_2 + \ldots  +Q'_{n\overline 1} l_n & =&  0 \\
% \\
% Q'_{1\overline 2}  l_1    +Q'_{2\overline 2} l_2 + \ldots  +Q'_{n\overline 2} l_n & = & 0\\
%  \vdots & \vdots & \vdots \\
% Q'_{1\overline n}  l_1    +Q'_{2\overline n} l_2 + \ldots  +Q'_{n\overline n}  l_n & = & 0 \\ 
%  \end{array}
% \right.
% \end{equation*}
which implies that each $\lambda_j$ 
vanishes identically since the determinant of the system is $Q'\neq 0$. 
From this we obtain  immediately that each $\mu_j$ also vanishes identically.
In summary, the arguments above show that either $\kappa_{2n}\geq 0 $ or $\lambda_j = \mu_j= 0$ for all $j=1,\cdots, n$. Since $\Theta$ is invertible, 
the latter would be a contradiction, hence we conclude that $\kappa_{2n} \geq 0$. This proves Lemma \ref{lemsurj}.
\end{proof}

\begin{lemma}\label{lemker}
The kernel of the linear operator $L_3 :(\mathcal A^{k,\alpha})^{2n} \to (\mathcal R_0)^{2n}$ has finite real dimension less than or equal to $2n(2k_0 - d)+2n$. 
\end{lemma}

\begin{proof}[Proof of Lemma \ref{lemker}]
According to Theorem 2.1 \cite{be-de2} (with $m=0$), the dimension of $\ker  L_3$ is equal to $\kappa+2n$, 
where $\kappa$ is the Maslov index of  
$\overline{A^{-1}}A$, namely
$${\rm ind} \det\left(-\overline{A^{-1}}A\right)=\frac{1}{2i\pi}\int_{b\Delta}\frac{\left[\det \left(-\overline{A(\zeta)}^{-1}A(\zeta)\right)\right]'}{\det \left(-\overline{A(\zeta)}^{-1}A(\zeta)\right)}{\rm d}\zeta.$$ 
Using (\ref{eqdet}) we have 
$$\det \overline{A^{-1}}A = (-1)^n\frac{Q'(\zeta)}{\overline{ Q'(\zeta)}}=(-1)^n\zeta^{-2(m_1+\cdots+m_n)}\frac{Q(\zeta)}{\overline{ Q(\zeta)}}.$$
 Therefore 
\[ \begin{aligned}
 {\rm ind} \det\left(-\overline{A^{-1}}A\right) &= -2\sum_{i=1}^n m_i + 2 {\rm ind} Q \\
 & \leq  -2\sum_{i=1}^n m_i+2 \left(n(2k_0 - d) + \sum_{i=1}^n m_i\right)=2n(2k_0 - d).
 \end{aligned}\]
\end{proof}

Finally, according to Lemma \ref{lemsurj}, Lemma \ref{lemker} in the present paper and Theorem 2.2 in \cite{be-de2}, the linear operator $L_2$ defined in (\ref{eqL2}) is onto and its kernel has finite real dimension $N$ less than or equal to $2k_0+2n(2k_0 - d)+2n+2=2(n+1)(k_0+1)+2nk_0-2dn$. This concludes the proof of Theorem
\ref{theodiscs2}.
\end{proof}

\subsection{The case of  homogeneous hypersurfaces}

Consider now the case of a model hypersurface defined as $S_P = \{\rho=0\} = \{-\real w + P(z,\overline z)=0 \}$ with $P$ a polynomial written as in (\ref{e:polydecomp}) of even degree $d$ and $m_1=m_2=\ldots=m_n=1$, that is
\[P(z, \bar z) =  \sum_{\substack{|J| + |K| = d \\ d-k_0\leq |J| \leq k_0}}\alpha_{JK} z^J \overline z^K. \]
In this situation we just say that $S_P$ is a homogeneous (rather than weighted homogeneous) hypersurface. We will assume the existence of an admissible vector in the sense of Definition \ref{defadm1}.

  The method followed in the previous section for the proof of Theorem \ref{theodiscs2} does not apply directly to $S_P$. In particular, in order to define the operator $L_2$ in Equation (\ref{eqL2}) one needs the weights $m_j$ to be even. However a slight modification of the procedure is possible: we apply the same rescaling as before to the system (i.e. we divide every line except the first and the last one by $(1-\zeta)^{d-m_j}=(1-\zeta)^{d-1}$) and then we multiply every line except the first and the last by $\zeta^s$, where $s=\frac{d-2}{2}$.  By Lemma \ref{subspaces} the resulting linear operator is of the kind
\[L_2: (\mathcal A^{k,\alpha})^{2n+2} \to \mathcal R_1 \times (\mathcal R_1)^{2n} \times  \mathcal C^{k,\alpha}\]
and the corresponding matrix $G_2$ is still of the form considered in Theorem 2.1 and Theorem 2.2 of \cite{be-de2}. The proofs of Lemmas \ref{lemsurj}  and \ref{lemker} are essentially the same, and the estimate on the dimension of the kernel in Lemma \ref{lemker} can be given as  $2n(2k_0-d)$.

In fact, stronger assumptions on the geometry of $S_P$ allow to be more precise on the dimension of the kernel, since it is possible in some cases to determine the Maslov index of $Q$ exactly. For instance, the following assumption is analogous to the one considered in \cite{be-de1} for hypersurfaces of $\mathbb C^2$:

\begin{lemma}\label{indexQ}
Suppose that the Levi form $P_{z\overline z}$ is positive definite outside of $0$. Then the index of $Q$ is $n(k_0 - \frac{d}{2}+1)$.
\end{lemma}
\begin{proof}
For any homogeneous polynomial $P(z,\overline z)$ of degree $d$, denote by $Q_P(\zeta)$ the holomorphic polynomial obtained by applying the procedure of section \ref{polmod} to $P$. For a small $\epsilon\geq 0$ we define $P_\epsilon$ as
\[P_\epsilon(z,\overline z) = |z_1|^d + \ldots + |z_n|^d + \epsilon \vnorm{ z }^d.\]
Note that the Levi form of $P_\epsilon$ is positive definite outside $0$ if $\epsilon>0$. One can compute directly that $Q_{P_0}(\zeta) = C \zeta^{n(k_0 - \frac{d}{2}+1)}$ for a certain constant $C$, hence the index of $Q_{P_\epsilon}(\zeta)$ is equal to 
$n(k_0- \frac{d}{2}+1)$ for $\epsilon>0$ small enough. On the other hand, the set of the homogeneous polynomials $P$ of degree $d$ such that $P_{z\overline z}$ is positive definite outside $0$ is a connected (and indeed convex) subset of the space of the polynomials of degree $d$, and since $Q_P(\zeta)$ depends continuously on $P$ it follows that its index is constant on this set.
\end{proof}

Following the proof of Lemma \ref{lemker} we have that the dimension of $\ker L_3$ is given by the Maslov index of $\overline{A^{-1}}A$ (since Theorem 2.1 from \cite{be-de2} must be applied with $m=1$), which in this case is just $ 2 {\rm ind} Q =n(2k_0-d+2)$. Accordingly, the dimension $N$ in Theorem \ref{theodiscs2} can be computed exactly as $2k_0 + n(2k_0-d+2) +2=2(n+1)(k_0+1)-dn$, which is lower than the estimate given in the general case by roughly a factor of $2$.

\subsection{The case of decoupled hypersurfaces}\label{secdec}
For a subset $I=\{i_1,\cdots,i_l\}\subset\{1,\cdots,n\}$, we set $z_{I}=(z_{i_1},\cdots,z_{i_l}).$ Let $\{I_1,...,I_k\}$ be a 
partition of $\{1,\cdots,n\}$.  
We consider a 
model hypersurface $S_P$ of the form $S_P=\{\rho=0\}\subset \mathbb C^{n+1}$, where 
\begin{equation*}
\rho(z,w)=- \real w + P(z,\overline z)=- \real w +  \sum_{j=1}^{k} P_j(z_{I_j},\overline{z_{I_j}})
\end{equation*}
where $P_j:\C^{|I_j|} \to \C$, $j=1,\cdots k$, is a (real) weighted homogeneous polynomial of (vector) weight 
$M_j \in \N^{|I_j|}$ and (weighted) degree $d_j \in \N$ written as in (\ref{e:polydecomp}). We denote by $k_0^1, \cdots k_0^k$ the corresponding integers. We assume that there exists an admissible vector $v$ for $P$.  

Consider a real smooth hypersurface $S=\{r=0\} \subset \C^{n+1}$ allowed in the sense of Section 3.1, namely 
$$r(z,w) = \rho(z,w) +\sum_{j=1}^k \theta_j(z_{I_j},\imag w )$$
where $\theta_j$, $j=1,\cdots,k$ is of the form (\ref{eqallow}). In such case, following the proof 
of Theorem \ref{theodiscs2}, the differential of the corresponding map $\mathcal{H}$ at $(\rho,{\bm f^0})$ is block upper triangular 
after permutation of coordinates. In this case, the corresponding operator $L_2$ (see (\ref{eqL2})) is of the form considered in 
Theorem 2.2 of \cite{be-de2}. Therefore if $S$ is close enough to $S_P$ in the sense Section 3.1 then Theorem \ref{theodiscs2} applies 
and provides a Banach manifold of stationary discs of real dimension at most $\sum_{j=1}^k 2(|I_j|+1)(k_0^j+1)-2d_j|I_j|$. 

Note that in principle such a model can be directly treated as a weighted homogeneous hypersurface by choosing different weights. However, 
in such case, the Banach manifold of stationary discs provided by Theorem \ref{theodiscs2} is of much greater dimension than the one obtained by considering the model as decoupled.      
\subsection{Construction of $k_0$-stationary  discs for admissible hypersurfaces}

\begin{defi}\label{defadm}
Let $S\subset \C^{n+1}$ be a finitely smooth real hypersurface through $0 \in \C^{n+1}$, 
and assume that $T^c_0 S = \{ w = 0\}$; write $w = u+ iv$. 
We say that $S$ is admissible if for a (sufficiently smooth) defining function
 (and hence for all sufficiently smooth defining functions)
$r(z,\bar z, \real w , \imag w)$ for $S$ near $0$ we have 
\[ r_u (0) = -1, \quad r_{z^J \bar z^K s^\ell} (0) = \begin{cases}
J! K! \alpha_{J,K} & \ell = 0, M(J+K) = d \\ 
0 & M(J+K) + \ell < d. 
\end{cases}   \]
Equivalently, $S$ is admissible if any defining function may be locally written as
\[
r(z,w)=\rho(z,w)+O\left(|z|^{d+1}\right)+\imag w \ O\left(|z,\imag w|^{d-1}\right)
\]
where $\rho(z,w)=- \real w + P(z,\overline z)$, and $P(z,\overline z)$ is of the form (\ref{e:polydecomp}) and admits an admissible vector.   
\end{defi}
We remark that the preceding definition is {\em independent of the choice of defining function}. It is also independent 
of the choice of holomorphic coordinates as long as the linear tangential part (the ``$z$-part'') preserves 
the weights. Being an admissible hypersurface is therefore a geometric concept.

Here $O\left(|z|^{d+1}\right)$ and $O\left(|z,\imag w|^{d-1}\right)$ are understood to be weighted orders where $z_j,\overline{z_j}$ and $w$ have respective weights $m_j$, $m_j$ and $1$.     
The following lemma is obtained exactly as Lemma $5.2$ in \cite{be-de1}: 
\begin{lemma}\label{lemdil}
Let $S\subset \C^{n+1}$ be an admissible real hypersurface of class $\mathcal{C}^{d+k+4}$.   
Consider the scaling $\Lambda_t(z,w)=(t^{m_1}z_1,\cdots,t^{m_n}z_n,t^dw)$.  For $t>0$ small enough, the defining function  
$\displaystyle r_t = \frac{1}{t^d} r \circ \Lambda_t$ belongs to the neighborhood $V$ in Theorem \ref{theodiscs2}.
\end{lemma}
Our main result about existence of discs follows now directly from the previous lemma  and Theorem \ref{theodiscs2}. 
 \begin{theorem}\label{theoadm}
Let $S\subset \C^{n+1}$ be an admissible real hypersurface of class $\mathcal{C}^{d+k+4}$. 
There exists a finitely dimensional biholomorphically invariant manifold of small $k_0$-stationary discs of class $\mathcal{C}^{k,\alpha}$ attached to $S$.
 \end{theorem}
\begin{remark}
In case $S$ is admissible with $P$ homogeneous or decoupled, the corresponding versions of Theorem \ref{theoadm} is sharper and provides a family of discs  of smaller 
dimension.  Note that for the decoupled case, the scaling $\Lambda_t$  should be modified; more precisely, following notations of Section \ref{secdec}, for $i \in I_j$, the variable $z_i$ 
must be scaled by $t^{m_i \Pi_{l \neq j} d_l}$.  
\end{remark}

\section{Finite jet determination of CR maps} 
\subsection{Statement of the result} 
The existence of $k_0$-stationary discs obtained in Theorem \ref{theoadm} allows us 
to obtain finite jet determination results for CR diffeomorphisms, generalizing the 
result from \cite{be-de1} to higher dimension.

\begin{theorem}\label{theo2} Let $P(z,\bar z)$ be 
a weighted homogeneous polynomial, of degree $d$. 
Then there exists an integer $\ell_0 \leq 6 n d$ 
such that the following holds. 
Let $S \subset \mathbb C^{n+1}$ be  an admissible real hypersurface of class $\mathcal{C}^{d+\ell_0 + 4}$ through $0 \in \C^{n+1}$, with model $S_P$.  
  If $H$ is a germ of a CR diffeomorphism of class $\mathcal C^{\ell_0 + 1}$ of $S$ satisfying $j^{\ell_0 +1}_0 H = I$, then $H = {\rm id}$.
\end{theorem}

Theorem~\ref{theo2} implies immediately 
Theorem~\ref{thm:main2}, which in conjunction with Lemma~\ref{lem:admissiblepoints} implies Theorem~\ref{thm:main}.
We will see how $\ell_0$ can be chosen in Lemma \ref{lemjet}. However, the intention of this paper is not to
give optimal bounds on the jet order needed for determination. This can be done better by considering purely formal 
constructions.

\begin{remark}
Assume that a jet determination result of order $k'$ holds in the formal setting, in the sense that every $\ell$-jet of a formal biholomorphisms which preserves a formal hypersurface  (up to the order $\ell$) and is trivial up to order $k'$ necessarily
coincides with the $\ell$-jet of the 
identity map. Then the conclusion of Theorem \ref{theo2} holds for $k'$-jet determination as long as the smoothness of $S$ is at least $\mathcal{C}^{\max\{k', d+\ell_0 + 4\}}$. Indeed, the $(\ell_0 + 1)$-order Taylor expansion of $H$ represents a $(\ell_0 + 1)$-order biholomorphism jet which preserves the polynomial hypersurface  induced by the Taylor expansion of $S$ up to order $(\ell_0 + 1)$, thus if it is trivial up to order $k'$ it must be trivial up to order $\ell_0+ 1$: from the theorem it follows in turn that $H$ is the identity.

It follows for instance that, for the version of Theorem \ref{theo2} in $\mathbb C^2$ (see Theorem 1.2 in \cite{be-de1}), we can always achieve $2$-jet determination of CR diffeomorphisms as in the real-analytic case (see \cite{eb-la-za,ko-me}). In higher dimension we can achieve the order of jet determination established in the formal setting, see for instance \cite{la-ju,la-mi} and 
for the model case \cite{ko-me2}.
\end{remark}
The proof of Theorem \ref{theo2} is achieved by putting together several facts, following the
approach taken in \cite{be-de1}:
\begin{itemize}
\item[1.] According to Proposition \ref{statinvar}, the family of $k_0$-stationary discs is invariant under CR diffeomorphisms.
\item[2.] By Lemma \ref{lemdil}, the pullback $r_t$ of the local defining function $r$ of $S$ under a suitable scaling method $\Lambda_t$ belongs to the neighborhood $V$ in Theorem \ref{theodiscs2}.
\item[3.] Similarly, the pullback $H_t = \Lambda_t^{-1} \circ H \circ \Lambda_t$ of the CR diffeomorphism $H$ can be made arbitrarily close to the identity (in the $\mathcal{C}^1$-norm) for $t$ small enough.
\item[4.] There exist an integer $\ell_0$, such that  the lifts of $k_0$-stationary discs attached to $r_t$ and passing through $0$ are determined by their $\ell_0$-jet at $1$.
\item[5.] The union of the images of $k_0$-stationary discs obtained in Theorem \ref{theodiscs2} is an open set of $\mathbb C^{n+1}$.
\end{itemize}

Similarly to Lemma \ref{lemdil} which is obtained exactly as Lemma $5.2$ in \cite{be-de1}, the point 3. is proved in 
the same way as Lemma 5.3 in \cite{be-de1}. To prove point 4., note that it is sufficient to show that the restriction of $\mathfrak j_{\ell_0}$ to the tangent space $T_{\bm {f^0}}\mathcal S^{k_0,\rho}$ of $\mathcal S^{k_0,\rho}$ at the point $\bm {f^0}=(f^0,\widetilde f^0)$ is injective: the statement then follows from Theorem \ref{theodiscs2}. Recall that here $\mathcal S^{k_0,\rho}$ denotes the set of lifts $\bm{f} \in Y^{M,d}$ (see (\ref{eqdefY2})) of $k_0$-stationary discs for the model hypersurface $\{\rho=0\}$ (see Definition \ref{defadm}). Since by the implicit function theorem $T_{\bm {f^0}}\mathcal S^{k_0,\rho}$ is kernel of the operator $\bm{f}'\mapsto 2\real \left[\overline{G(\zeta)}\bm{f}'\right]$ (see \ref{eqrh0}), the claim is a consequence of Lemma \ref{lemjet} proved in the next section. We will prove point 5. in Lemma \ref{lem:opencover}.

Finally, the proof of Theorem \ref{theo2} follows from the points above with the same argument as in Section 5.2 of \cite{be-de1}: the only difference is that one needs to apply the argument to the lift of $H_t$ to the conormal bundle rather than to $H_t$ itself, and this is achieved as in Section 4.2 \cite{be-bl}.

\subsection{Injectivity of the jet map} 
Let $\ell_0, m, N\in \mathbb N$. 
We want to consider the linear map $\mathfrak j_{\ell_0}: Y^{M,d} \to \mathbb C^{(2n+2)(\ell_0 + 1)}$ sending 
${\bm f}$ to its $\ell_0$-jet at $\zeta=1$
\[ \mathfrak j_{\ell_0}({\bm f})=\left ( {\bm f}(1), \partial {\bm f}(1), \ldots, \partial_{\ell_0} {\bm f}(1)\right )\in  \mathbb C^{N(\ell_0 + 1)} \]
where $\partial_\ell {\bm f}(1) \in \mathbb C^N$
denotes  the vector $\displaystyle \frac{\partial^\ell {\bm f}}{\partial \zeta^\ell}(1)$ for all $\ell=1,\cdots,\ell_0$. 
\begin{lemma}\label{lemjet}
There exists an integer $\ell_0\leq 6nd$ such that the restriction of 
$\mathfrak j_{\ell_0}$ to the kernel of the operator $\bm{f}'\mapsto 2\real \left[\overline{G(\zeta)}\bm{f}'\right]$ (see \ref{eqrh0}) is injective.
\end{lemma}
\begin{proof}
Following the notation of the proof of Theorem \ref{theodiscs2}, we prove that there exists an integer 
$\ell_0$ such that the restriction of $\mathfrak j_{\ell_0}$ to the kernel of $L_2$ (see (\ref{eqL2})) is injective.  
According to  Lemma 5.1 in \cite{gl1} we write  
\begin{equation*}
-\overline {G_2^{-1}} G_2 = \Theta_2^{-1} \Lambda \overline \Theta_2
\end{equation*}
where $\Theta_2: \overline{\Delta} \to GL_{2n+2}(\C)$ is a  smooth map holomorphic on $\Delta$, and $\Lambda$  is the diagonal matrix with entries $\zeta^{k_1}, \ldots, \zeta^{k_{2n+2}}$ where $k_1, \cdots k_{2n+2}$ are the partial indices of $\overline {G_2^{-1}} G_2$. Let ${\bm f} \in \ker L_1$. We can write 
$$
{\bm f}=-\overline {G_2^{-1}} G_2 \overline{{\bm f}}=\Theta_2^{-1} \Lambda \overline \Theta_2\overline{{\bm f}}
$$   
and therefore 
$$\Theta_2{\bm f}= \Lambda \overline{\Theta_2 {\bm f}}.$$
It follows that the $j^{th}$-component of $\Theta_2{\bm f}$ is a polynomial of degree at most $k_j$. Hence 
$\Theta_2{\bm f}$ is determined by its $\ell_0:=\max\{k_1,\cdots,k_{2n+2}\}$-jet at $1$. It remains to prove that    
 the restriction of $\mathfrak j_{\ell_0}$ to $\ker L_1$ is injective.  
Indeed, for any $\ell\geq 0$ we have
\[\partial_\ell(\Theta_2 {\bm f})(1) = \Theta_2(1) \partial_\ell {\bm f}(1) + R_{\ell - 1}\]
where $R$ is a linear function of the $(\ell - 1)$-jet of $f$ at $1$. It follows that the (well-defined) linear map $\Theta_{\ell_1}: \mathbb C^{(2n+2)(\ell_0 + 1)}\to  \mathbb C^{(2n+2)(\ell_0 + 1)}$ which sends the $\ell_0$-jet of ${\bm f}$ at $1$ to the $\ell_0$-jet of $\Theta_2 {\bm f}$ at $1$ has a block-triangular matrix representation whose $(2n+2)\times (2n+2)$ blocks in the diagonal are all equal to the non-singular matrix $\Theta_2(1)$. Therefore 
$\Theta_{\ell_2}$ is invertible, and the claim follows from the fact that $\mathfrak j_{\ell_0} \circ \Theta_2 = \Theta_{\ell_2} \circ \mathfrak j_{\ell_0}$ and that $\mathfrak j_{\ell_0} $ is injective on $\Theta_2(\ker L_1)$.
To conclude the proof we estimate $\ell_0$ by the Maslov index of $-\overline {G_2^{-1}} G_2$, namely 
\[\begin{aligned}
{\rm ind} \det\left(-\overline {G_2^{-1}} G_2\right) &= -2\sum_{i=1}^n m_i + 2 {\rm ind} Q +2k_0 \\
 &\leq 4n(2k_0 - d)) + \sum_{i=1}^n m_i+2k_0 \leq 6nd.
 \end{aligned}\]
\end{proof}

\subsection{An extended family of discs; covering of an open subset} 
\label{sub:a_special_family_of_stationary_discs}
We choose an allowable vector $v$ as described  subsection \ref{polmod}, 
and consider the disk $f^v = (h^v, g^v)$ associated with it. 
This disk is $k_0$-stationary, since
\[ \partial \rho \circ f^v = \left( P_{z_1} (h^v, \bar h^v), \dots , P_{z_n} (h^v, \bar h^v), -\frac{1}{2} \right), \]
and the degree in $\bar \zeta$ of each of the components is at most $k^0$; hence
$\zeta^{k_0} \partial \rho \circ f^v$ does extend holomorphically 
to $\Delta$. Consider, for every $a\in \Delta$, also 
the disk $f^v_a = f^v \circ \varphi_a$, where
\[ \varphi_a (\zeta) = \frac{1- \bar a}{1- a} \frac{\zeta - a}{1 - \bar a \zeta}.\]
This extended family of disks is 
useful, because we can compute the rank of its 
center evaluation map $(v,a)\mapsto C(v,a) = f^v_a (0) = (v, g^v_a (0)) $. By construction, 
the (real) Jacobian of this map at $(v,a) = (v,0)$ is  given by 
\[ \begin{aligned}
\det \begin{pmatrix}
	\dop{a}\big|_0 g^v_a (0) & \dop{\bar a}\big|_0 g^v_a (0) \\ \hspace{2mm} \\
	\dop{a}\big|_0 \overline{g^v_a (0)} & \dop{\bar a}\big|_0 \overline{g^v_a (0)} \\
\end{pmatrix}
&=\det \begin{pmatrix}
	(g^v)' (0) \dop{a}\big|_0 \varphi_a (0) & (g^v)' (0) \dop{\bar a}\big|_0 \varphi_a (0) \\ \hspace{2mm} \\
	\overline{(g^v)' (0)} \dop{a}\big|_0 \overline{\varphi_a (0)} & \overline{(g^v)' (0)} \dop{\bar a}\big|_0 \overline{\varphi_a (0)} \\
\end{pmatrix} \\ & = \det \begin{pmatrix}
	- (g^v)' (0)  &  0 \\ \hspace{2mm} \\
	0 & -  \overline{(g^v)' (0)} \\
\end{pmatrix} \\ & = |(g^v)' (0)|^2.
\end{aligned}  \]
We therefore have that the center evaluation map $(v,a) \mapsto C(v,a)$ is of 
full rank at $(v,a) = (v_0, 0)$ if and only if $(g^{v_0})' (0) \neq 0$. 

However, we can also {\em compute} $g^v$: $g^v$ is the holomorphic function
which satisfies $g^v (1) = 0$ and $\real g^v(\zeta) = P(h^v(\zeta), \overline{h^v (\zeta)})$
if $\zeta \bar \zeta =1$. Therefore, 
\[ \begin{aligned}
\real g^v(\zeta) &= \real \sum_{j=d-k_0}^{k_0} (1-\zeta)^j (1 - \bar \zeta)^{d-j} P^{j,d-j} (v, \bar v) \\ 
& = \real \sum_{j=d-k_0}^{k_0}  \left(\sum_{\ell} \binom{j}{\ell}\binom{d-j}{\ell} \right) P^{j,d-j} (v, \bar v)\\ & \mbox{ } \ \ +2 \real \sum_{j=d-k_0}^{k_0} \sum_{e=1}^{|d-2j|} (-1)^e \zeta^e \left(\sum_{\ell} \binom{j}{e + \ell}\binom{d-j}{\ell} \right) P^{j,d-j} (v, \bar v).
\end{aligned}
\]
From this equality it is easy to see that 
\[ (g^v)'(0) = -  2 \sum_{j=d-k_0}^{k_0}  \left(\sum_{\ell} \binom{j}{1 + \ell}\binom{d-j}{\ell} \right) P^{j,d-j} (v, \bar v).  \]
Hence, $(g^v)' (0) \neq 0$ for a dense, open subset of the $v$'s.

In particular, since the image of the model stationary disc $f^v$
is the same as the image of the $f^v_a$ for $a \in \Delta$, we have 
the following
\begin{lemma}
\label{lem:opencover} 
The set $\cup_v f^v (\Delta)$ contains an open subset of $\C^{n+1}$. 
\end{lemma}

\vskip 0,1cm
{\small
\noindent Florian Bertrand\\
Department of Mathematics\\
American University of Beirut, Beirut, Lebanon\\{\sl E-mail address}: fb31@aub.edu.lb\\

\noindent Giuseppe Della Sala \\
Department of Mathematics\\
American University of Beirut, Beirut, Lebanon\\{\sl E-mail address}: 	gd16@aub.edu.lb\\

\noindent Bernhard Lamel \\
Department of Mathematics, University of Vienna, Oskar-Morgenstern-Platz 1, Vienna, 1090, Austria\\
{\sl E-mail address}: 	bernhard.lamel@univie.ac.at\\

}


\begin{thebibliography}{11111} 

\bibitem{ber}M.S. Baouendi, P. Ebenfelt, L.P. Rothschild, {\it Real submanifolds in complex space and their mappings}, 
Princeton Mathematical Series, {\bf 47}. Princeton University Press, Princeton, NJ, 1999. xii+404 pp.

\bibitem{bmr}M.S. Baouendi, N. Mir, and L. Rothschild, {\em Reflection ideals and mappings between generic submanifolds in complex space}. J. Geom. Anal. 12 (2002), no. 4, 543–580. 

\bibitem{bl} L. Blanc-Centi, {\it Stationary discs glued to a Levi non-degenerate hypersurface}, Trans. Amer. Math. Soc. {\bf 361} (2009), 3223-3239. 
\bibitem{be-bl} F. Bertrand, L. Blanc-Centi, {\it Stationary holomorphic discs and finite jet determination problems},  
Math. Ann. {\bf 358} (2014), 477-509.

\bibitem{be-de1} F. Bertrand, G. Della Sala, {\it Stationary discs for smooth hypersurfaces of finite type and finite jet determination},  J. Geom. Anal. {\bf 25} (2015),
 2516-2545.

\bibitem{be-de2} F. Bertrand, G. Della Sala, {\it Riemann-Hilbert problems with constraints}, preprint.

\bibitem{ce} M. \v{C}erne, {\it Analytic discs attached to a generating CR-manifold}, Ark. Mat. {\bf 33} (1995), 217-248.

\bibitem{ch-mo} S.S. Chern, J.K. Moser, {\it Real hypersurfaces in complex manifolds}, 
Acta math. {\bf 133} (1975), 219-271.


\bibitem{eb} P. Ebenfelt, {\it Finite jet determination of holomorphic mappings at the boundary}, 
Asian J. Math. {\bf 5} (2001), 637-662.


\bibitem{eb-la}P. Ebenfelt, B. Lamel, {\it Finite jet determination of CR embeddings}, 
J. Geom. Anal. {\bf 14} (2004), 241-265.


\bibitem{eb-la-za} P. Ebenfelt, B. Lamel, D. Zaitsev, {\it Finite jet determination of local analytic CR automorphisms and their parametrization by 2-jets in the finite type case}, Geom. Funct. Anal.
 {\bf 13} (2003), 546-573.

\bibitem{fo} F. Forstneri\v{c}, {\it Analytic disks with boundaries in a maximal real submanifold of $\C^2$}, Ann. Inst. Fourier {\bf 37} (1987), 1-44.


\bibitem{gl1}J. Globevnik, {\it Perturbation by analytic discs along maximal real submanifolds of $\C^N$},
 Math. Z. {\bf 217} (1994), 287-316.

\bibitem{gl2}J. Globevnik, {\it Perturbing analytic discs attached to maximal real submanifolds of $\C^N$}, 
Indag. Math.  {\bf 7} (1996), 37-46.


\bibitem{hi-ta}C.D. Hill, G. Taiani, {\it Families of analytic discs in $\C^n$ with boundaries on a prescribed CR submanifold},
 Ann. Scuola Norm. Sup. Pisa Cl. Sci. (4) {\bf 5} (1978), 327-380.
 
 
 
\bibitem{hu3}X. Huang, {\it A preservation principle of extremal mappings near a strongly pseudoconvex point and
 its applications}, Illinois J. Math. {\bf 38} (1994), 283-302. 

\bibitem{hu}X. Huang, {\it A non-degeneracy property of extremal mappings and iterates of holomorphic 
self-mappings}, Ann. Scuola Norm. Sup. Pisa Cl. Sci. (4) {\bf 21} (1994), 399-419. 

\bibitem{ju}  Juhlin, Robert, {\em Determination of formal CR mappings by a finite jet}, Adv. Math. 222 (2009), no. 5, 1611–1648.
 
\bibitem{ki-za} S.-Y. Kim,  D. Zaitsev, {\it Equivalence and embedding problems for CR-structures of any 
codimension}, Topology {\bf 44} (2005), 557-584.


\bibitem{ko-me}M. Kol\'a\v{r}, F. Meylan, {\it Infinitesimal CR automorphisms of hypersurfaces of finite type in $\C^2$}, 
Arch. Math. (Brno) {\bf 47} (2011), 367-375.

\bibitem{ko-me2} M. Kol\'a\v{r}, F.  Meylan, D. Zaitsev. {\it Chern-Moser operators and polynomial models in CR geometry}.
 Adv. Math.  263  (2014), 321--356.

\bibitem{lasurvey} B. Lamel. {\em Jet embeddability of local automorphism groups of real-analytic CR manifolds}. In: Geometric analysis of several complex variables and related topics, 89–108, Contemp. Math., 550, Amer. Math. Soc., Providence, RI, 2011.

\bibitem{la-ju} R. Juhlin,  B. Lamel, {\it Automorphism groups of minimal real-analytic CR manifolds}. Journal of the European Mathematical Society (JEMS), 15(2), 509–-537. %http://doi.org/10.4171/JEMS/366

\bibitem{la-mi} B. Lamel, N. Mir, {\it Finite jet determination of local CR automorphisms through 
resolution of degeneracies}, Asian J. Math. {\bf 11} (2007), 201-216.

\bibitem{le}L. Lempert, {\it La m\'etrique de Kobayashi et la repr\'esentation des domaines sur la boule}, Bull. Soc. Math. France 
{\bf 109} (1981), 427-474.

% \bibitem{oh}Y.-G. Oh, {\it Riemann-Hilbert problem and application to the perturbation theory of analytic discs}, Kyungpook Math. J. {\bf 35} 
% (1995), 39-75.


% \bibitem{ro}P. Rossi, {\it On the defect of an analytic disc}, Ann. Mat. Pura Appl.  {\bf 173} (1997), 333-349.


\bibitem{tu} A. Tumanov, {\it Extremal discs and the regularity of CR mappings in higher codimension}, Amer. J. Math. 
{\bf 123} (2001), 445-473.


\bibitem{ve}N.P. Vekua, {\it Systems of singular integral equations}, Noordhoff, Groningen (1967) 216 pp.


\bibitem{we}S. Webster, {\it On the reflection principle in several complex variables}, 
Proc. Amer. Math. Soc. {\bf 71} (1978), 26-28.

\end{thebibliography}
\end{document}